\documentclass[a4paper,10pt,reqno, english]{amsart}
\usepackage[utf8]{inputenc}
\usepackage[T1]{fontenc}
\usepackage{amsmath,amsthm}
\usepackage{amsfonts,amssymb,enumerate}
\usepackage{url,paralist}
\usepackage{mathtools,color}
\usepackage[colorlinks=true,urlcolor=blue,linkcolor=red,citecolor=magenta]{hyperref}
\usepackage{enumerate}
\usepackage{anysize}
\usepackage[arrow,curve,matrix,tips,2cell]{xy}
  \SelectTips{eu}{10} \UseTips
  \UseAllTwocells
\usepackage{lscape}
\graphicspath{{fig/}}

\theoremstyle{plain}
\newtheorem{theorem}{Theorem}[section]

\newtheorem{corollary}[theorem]{Corollary}
\newtheorem{proposition}[theorem]{Proposition}

\newtheorem*{theorem*}{Theorem}

\newtheorem*{claim*}{Claim}
\newtheorem*{lemma*}{Lemma}

\theoremstyle{definition}

\newtheorem{example}[theorem]{Example}

\newcommand\RP{\mathbb{R}{\rm P}}
\newcommand{\R}{\mathbb{R}}

\newcommand{\OO}{\mathrm{O}}
\newcommand{\SO}{\mathrm{SO}}

\newcommand{\Z}{\mathbb{Z}}

\newcommand{\F}{\mathbb{F}}
\newcommand{\B}{\mathrm{B}}
\newcommand{\E}{\mathrm{E}}
\newcommand{\pt}{\mathrm{pt}}
\DeclareMathOperator{\im}{im}

\newcommand{\id}{\operatorname{id}}
\newcommand{\ind}{\operatorname{Index}}

\newcommand{\res}{\operatorname{res}}

\begin{document}

\title{Index of Grassmann manifolds and orthogonal shadows}



\author[Barali\'c]{Djordje Barali\'c}
\thanks{The research by Djordje Barali\'c leading to these results has
        received funding from the Grant 174020 of the Ministry for Education and Science of the Republic of
Serbia.}
\address{Mat. Institut SANU, Knez Mihailova 36, 11001 Beograd, Serbia}
\email{djbaralic@mi.sanu.ac.rs}

\author[Blagojevi\'c]{Pavle V. M. Blagojevi\'{c}}
\thanks{The research by Pavle V. M. Blagojevi\'{c} leading to these results has received funding from DFG via the Collaborative Research Center TRR~109 ``Discretization in Geometry and Dynamics'', and the grant ON 174008 of the Serbian Ministry of Education and Science.
This material is based upon work supported by the National Science Foundation under Grant No. DMS-1440140 while Pavle V. M. Blagojevi\'{c} was in residence at the Mathematical Sciences Research Institute in Berkeley, California, during the Fall of 2017.}
\address{Inst. Math., FU Berlin, Arnimallee 2, 14195 Berlin, Germany\hfill\break%
\mbox{\hspace{4mm}}Mat. Institut SANU, Knez Mihailova 36, 11000 Beograd, Serbia}
\email{blagojevic@math.fu-berlin.de}

\author[Karasev]{Roman Karasev}
\thanks{The research by Roman~Karasev leading to these results has received funding from the Federal professorship program grant 1.456.2016/1.4, the Russian Science Foundation grant 18-11-00073, and the Russian Foundation for Basic Research grant 18-01-00036.}
\address{Moscow Institute of Physics and Technology, Institutskiy per. 9, Dolgoprudny, Russia 141700\hfill\break
\mbox{\hspace{4mm}}Institute for Information Transmission Problems RAS, Bolshoy Karetny per. 19, Moscow, Russia 127994}
\email{r\_n\_karasev@mail.ru}

\author[Vu\v{c}i\'c]{Aleksandar Vu\v{c}i\'c}
\address{University of Belgrade, Faculty of Mathematics, Studentski Trg 16, 11000 Beograd, Serbia}
\email{avucic@matf.bg.ac.rs}

\date{\today}

\maketitle

\begin{abstract}
In this paper we study the $\Z/2$ action on real Grassmann manifolds $G_{n}(\R^{2n})$ and $\widetilde{G}_{n}(\R^{2n})$ given by taking (appropriately oriented) orthogonal complement.
We completely evaluate the related $\Z/2$ Fadell--Husseini index utilizing a novel computation of the Stiefel--Whitney classes of the wreath product of a vector bundle.
These results are used to establish the following geometric result about the orthogonal shadows of a convex body: For $n=2^a (2 b+1)$, $k=2^{a+1}-1$, $C$ a convex body in $\R^{2n}$, and $k$ real valued functions $\alpha_1,\ldots,\alpha_k$ continuous on convex bodies in $\R^{2n}$ with respect to the Hausdorff metric, there exists a subspace $V\subseteq\R^{2n}$ such that projections of $C$ to $V$ and its orthogonal complement $V^{\perp}$ have the same value with respect to each function $\alpha_i$, which is  $\alpha_i (p_V(C))=\alpha_i (p_{V^\perp} (C))$ for all $1\leq i\leq k$.

\end{abstract}


\section{Introduction}
\label{sec : Introduction and the statement of the main results }

The Grassmann manifold of all $n$ dimensional linear subspaces in the vector space $V$ over some field is one of the classical and widely studied objects of algebraic topology with important applications in differential geometry and algebraic geometry.
In this paper we study a particular free $\Z/2$ action on a real Grassmann manifolds induced by taking orthogonal complements and use its properties to present an interesting geometric application.
The main tool in the study of this action is the ideal valued index theory of Fadell and Husseini.

\medskip
A variety of different index theories, for the study of the non-existence of equivariant maps, were introduced and considered over the last seven decades.
In 1952 Krasnosel'ski\u{i} \cite{Krasnoselski1952} introduced a numerical $\Z/2$-index for the subsets of spheres, which he called {\em genus}.
He used genus to estimate the number of critical points of a particular class of weakly continuous functionals on a spheres in a Hilbert space; consult also \cite{Krasnoselski1984}.
Yang in 1955 introduced yet another $\Z/2$-index calling it {\em B-index} with the aim of obtaining results of Borsuk--Ulam type \cite{Yang1955}.
The work of Conner and Floyd followed at the beginning of '60s and the {\em co-index} was introduced \cite{Conner1960,Conner1962}.
At the end of '80s Fadell and Husseini in series of papers introduced a general  index theory and applied it on various problems \cite{Fadell1985,Fadell1987,Fadell1987-02,Fadell1988,Fadell1989}.
Further refinements of the Fadell and Husseini work were done by Volovikov in \cite{Volovikov2002}.
Since then computation of a concrete index for the given space and the given action was, and still is, a formidable, and by now a classical, problem.
For example, the work of 
Volovikov \cite{Volovikov1992}, 
Petrovi\'c \cite{Petrovic1997}, 
Hara \cite{Hara2005}, 
Crabb \cite{Crabb2012},
Blagojevi\'c and Karasev \cite{blagojevic-karasev2012}, 
Petrovi\'{c} and  Prvulovi\'{c} \cite{Petrovic2013},
Blagojevi\'c, L\"uck and Ziegler \cite{Blagojevic2012},
Simon \cite{Simon2016}
illustrate wide diversity of novel ideas and methods needed to be employed in course of computation of different Fadell--Husseini indexes.  

\medskip
Let $n\geq 1$ and $1\leq k\leq n$ be integers.
The {\em real Grassmann manifold} of all $n$ dimensional linear subspaces in the Euclidean space $\R^{n+k}$ is denoted by $G_{n}(\R^{n+k})$.
Classically, as a homogeneous space, it is defined to be the quotient
\[
G_{n}(\R^{n+k}):=\OO(n+k)/(\OO(n)\times\OO(k)),
\]
where $\OO(n)$ denotes the orthogonal group.
The {\em real oriented Grassmann manifold} of all oriented $n$ dimensional linear subspaces in the Euclidean space $\R^{n+k}$ is denoted by $\widetilde{G}_{n}(\R^{n+k})$, and is defined to be
\[
\widetilde{G}_{n}(\R^{n+k}):=\SO(n+k)/(\SO(n)\times\SO(k)),
\]
where $\SO(n)$ denotes the special orthogonal group.
The inclusion map of the pairs
\[
(\SO(n+k),\SO(n)\times\SO(k))\hookrightarrow (\OO(n+k),\OO(n)\times\OO(k))
\]
induces the map of quotients $c\colon \widetilde{G}_{n}(\R^{n+k})\longrightarrow G_{n}(\R^{n+k})$ which is a double cover.

\medskip
Now we specialize to the situation where the integers $k$ and $n$ coincide.
Let $\Z/2=\langle\omega\rangle$ be the cyclic group of order $2$ generated by $\omega$.
A $\Z/2$ action on real Grassmann manifolds $G_{n}(\R^{2n})$ and $\widetilde{G}_{n}(\R^{2n})$  we consider is defined by sending an $n$-dimensional (oriented) subspace $V$ to its (appropriately oriented) orthogonal complement $V^{\perp}$, that means $\omega\cdot V= V^{\perp}$.
 In the case when $n$ is even these actions can be lifted to actions on $\OO(2n)$ and $\SO(2n)$ in such a way that the quotient maps $p_1\colon \OO(2n)\longrightarrow G_{n}(\R^{2n})$ and $p_2\colon \SO(2n)\longrightarrow \widetilde{G}_{n}(\R^{2n})$ are $\Z/2$-maps.
Indeed, for $(v_1,\ldots,v_n,v_{n+1},\ldots,v_{2n})$ and element of $\OO(2n)$, or $\SO(2n)$ we set
\[
\omega\cdot (v_1,\ldots,v_n,v_{n+1},\ldots,v_{2n})= (v_{n+1},\ldots,v_{2n},v_1,\ldots,v_n).
\]
Thus, for $n$ even we have the following commutative square of $\Z/2$-maps:
\[
\xymatrix{
\SO(2n)\ar[r]^-{p_2}\ar[d]^{i} & \widetilde{G}_{n}(\R^{2n})\ar[d]^{c}\\
\OO(2n)\ar[r]^-{p_1} & G_{n}(\R^{2n}),
}
\]
where $i\colon \SO(2n) \longrightarrow \OO(2n)$ is the inclusion, and $c\colon \widetilde{G}_{n}(\R^{n+k})\longrightarrow G_{n}(\R^{n+k})$ is the double cover map.

\medskip
The central result of this paper is the evaluation of the Fadell--Husseini index \cite{Fadell1988} of the Grasmmann manifolds with respect to the $\Z/2$ action we just introduced.
Before we state the result we briefly recall the definition and some basic properties of the Fadell--Husseini index.
For a finite group $G$ and a $G$-space $X$ the {\em Fadell--Husseini index} with the coefficients in the fields $\F$ is the  kernel of the following map in cohomology
\[
	\ind_G (X;\F) :=\ker\big( \pi_X^*\colon H^*(\mathrm{E}G\times_G\mathrm{pt};\F)\longrightarrow H^*(\mathrm{E}G\times_G X;\F)\big),
\]
where the map $\pi_X\colon \mathrm{E}G\times_GX\longrightarrow\mathrm{E}G\times_G\mathrm{pt}$ is induced by the $G$-map $X\longrightarrow\mathrm{pt}$.
The $G$-action on the point $\mathrm{pt}$ is assumed to be trivial.
The space $\mathrm{E}G\times_GX$ is known as the Borel construction of the $G$-space $X$.
There are natural isomorphisms  $ H^*(\mathrm{E}G\times_G\mathrm{pt};\F)\cong H^*(\mathrm{B}G;\F)\cong H^*(G;\F)$ and therefore we do not distinguish them.
For the cohomology of the group $\Z/2$ we fix the following notation
\[
H^*(\Z/2;\F_2)= H^*(\B(\Z/2);\F_2)=\F_2[t],
\]
with $\deg(t)=1$.
The essential property of the Fadell--Husseini index \cite[p.\,74]{Fadell1988} is that:
If $X$ and $Y$ are $G$-spaces and if there exists a continuous $G$-map $X\longrightarrow Y$ then
\[
\ind_G (X;\F)\supseteq \ind_G (Y;\F).
\]
For example, it is a known fact that the Fadell--Husseini index of the sphere $S^{n-1}$, equipped with any free $\Z/2$ action, is the ideal generated by $t^{n}$, that is $\ind_{\Z/2}(S^{n-1};\F_2)=\langle t^n\rangle$.

The Fadell--Husseini index of a path connected $G$-space $X$ with coefficients in a field $\F$ can be computed from the Serre spectral sequence associated with the Borel construction fibration
\[
\xymatrix{
X\ar[r] & \E G\times_G X\ar[r] & \B G,
}
\]
whose $E_2$-term is given by
\[
E_2^{i,j}(\E G\times_G X)=H^i(\B G; \mathcal{H}^j(X;\F))\cong H^i(G; H^j(X;\F)).
\]
Here $H^i(\B G; \mathcal{H}^j(X;\F))$ denotes the cohomology with local coefficients induced by the action of the fundamental group of the base space $\pi_1(\B G)\cong G$ on the cohomology of the fiber $H^j(X;\F)$.
On the other hand $H^i(G; H^j(X;\F))$ denotes the cohomology of the group $G$ with coefficients in the $G$-module $H^j(X;\F)$. Those two cohomologies are isomorphic by definition.
For a definition of cohomology with local coefficients consult for example \cite[Sec.\,3.H]{Hatcher2002}.
The Fadell--Husseini index of a path-connected $G$-space $X$ can be evaluated using the equality
\begin{equation}
	\label{eq : index from ss}
	\ind_G (X;\F) = \ker \big(E_2^{*,0}(\E G\times_G X)\longrightarrow E_{\infty}^{*,0}( \E G\times_G X)\big).
\end{equation}

\medskip
In this paper we compute the Fadell--Hussini index of Grassmann manifolds with respect to the introduced $\Z/2$ action and prove the following two theorems.
\begin{theorem}
	\label{th : main - 01}
	Let $n\geq 1$ be an integer, and let $a\geq 0$ and $b\geq 0$ be the unique integers such that $n=2^a(2b+1)$.
	Then
	\[
	\ind_{\Z/2}(\widetilde{G}_{n}(\R^{2n});\F_2)=\begin{cases}
		\langle t^3\rangle, & a=1,\\
		\langle t^{2^{a+1}}\rangle, & a=0 \text{ or } a\geq 2.
	\end{cases}
	\]
\end{theorem}

\begin{theorem}
	\label{th : main - 02}
	Let $n\geq 1$ be an integer, and let $a\geq 0$ and $b\geq 0$ be the unique integers such that $n=2^a(2b+1)$.
	Then
	\[\ind_{\Z/2}(G_{n}(\R^{2n});\F_2)=\langle t^{2^{a+1}}\rangle .\]
\end{theorem}

\medskip
Theorems \ref{th : main - 01} and \ref{th : main - 02} have interesting geometrical consequences.
One of them is the following result.
 
\begin{corollary}
\label{cor : geometric}
Let $a\geq 0$ and $b\geq 0$ be integers, $n=2^a (2 b+1)$ and $k=2^{a+1}-1$.
Let $\alpha_1,\ldots,\alpha_k$ be real valued functions on the space of $n$-dimesional convex bodies in $\R^{2n}$ that are continuous with respect to the Hausdorff metric.
Then for every proper convex body $C\subseteq\R^{2n}$ there exists an $n$-dimensional
subspace $V\subseteq\R^{2n}$ such that
\begin{equation}
	\label{eq : claim of corollary}
	\alpha_i (p_V(C))=\alpha_i (p_{V^\perp} (C))
	\qquad\qquad\text{for all }1\leq i\leq k,
\end{equation}
where $p_V$ and $p_{V^\perp}$ are orthogonal projections onto $V$ and its orthogonal complement $V^{\perp}$.
\end{corollary}

\begin{proof}
Let $\alpha_1,\ldots,\alpha_k$ be continuous real valued functions on the space of $n$-dimesional convex bodies in $\R^{2n}$, and let $C\subseteq\R^{2n}$ be a convex body.
Let $f_C\colon G_{n}(\R^{2n})\longrightarrow \R^k$ be  a continuous map defined
by
\[
f_C(V):=\big( \alpha_1
(p_V(C))-\alpha_1 (p_{V^\perp} (C)), \ldots, \alpha_k
(p_V(C))-\alpha_k (p_{V^\perp} (C))\big)
\]
for $V\in G_{n}(\R^{2n})$.
If the $\Z/2$ action on the Euclidean space $\R^k$ is assumed to be antipodal, and the $\Z/2$ action on $G_{n}(\R^{2n})$ is given by $\omega\cdot V=V^{\perp}$, then $f_C$ is a $\Z/2$-map.

If the claim \eqref{eq : claim of corollary} does not hold then there exists a convex body $C$ such that $0\not\in\mathrm{im} f_C$.
Thus the map $f_C$ factors through $\R^k \setminus \{0\}$  as follows
\[
\xymatrix{
G_{n}(\R^{2n})\ar[rr]^-{f_C}\ar[dr]_{\widetilde{f}_C} && \R^k\\
& \R^k \setminus \{0\}\ar[ur]_{i}.&
}
\]
There exists a composition of $\Z/2$-maps
\[
\xymatrix{
G_{n}(\R^{2n})\ar[r]^-{\widetilde{f}_C} & \R^k \setminus \{0\}\ar[r]^-{r} & S^{k-1},
}
\]
where $r$ is the $\Z_2$-invariant radial retraction onto the sphere $S^{k-1}$ with the antipodal action.
Consequently, by the basic property of the Fadell--Husseini index
\[
\langle t^{2^{a+1}}\rangle=\ind_{\Z/2}({G}_{n}(\R^{2n});\F_2)\supseteq \ind_{\Z/2}(S^{k-1};\F_2)= \langle
 t^{k}\rangle= \langle
 t^{2^{a+1}-1}\rangle.
\]
This is a contradiction, and therefore the claim \eqref{eq : claim of corollary} holds.

\end{proof}

Geometrically meaningful examples for functions $\alpha_1,\ldots,\alpha_k$  may be given by quermassintegrals (intrinsic volumes), that is average $k$-dimensional volumes of orthogonal $k$-dimensional projections of a body $C\subset\mathbb R^{2n}$, for $1\leq k\leq n$.
Another meaningful function is the radius of the minimal containing ball of $C$.
These functions have an additional property that they are rotationally and translationally invariant.

\medskip
A direct consequence of the proof of Corollary \ref{cor : geometric} is the following geometrical result.

\begin{corollary}
\label{cor : geometric-02}
Let $a\geq 0$ and $b\geq 0$ be integers, $n=2^a (2 b+1)$ and $k=2^{a+1}-1$.
Let $\alpha_1,\ldots,\alpha_k$ be real valued functions on the space of $n$-dimesional convex bodies in $\R^{2n}$ that are continuous with respect to the Hausdorff metric.
Then for every proper convex body $C\subseteq\R^{2n}$ containing the origin in its interior there exists an $n$-dimensional
subspace $V\subseteq\R^{2n}$ such that
\[
	\alpha_i (C\cap V)=\alpha_i (C\cap V^\perp)
	\qquad\qquad\text{for all }1\leq i\leq k,
\]
where $V^{\perp}$ is the orthogonal complement of $V$.
\end{corollary}

For example, if a proper convex body $C\subseteq \R^4$ with the origin in its interior is given, then there exist orthogonal $2$-dimensional planes through the origin $V$ and $V^{\perp}$ such that the intersections $C\cap V$ and $C\cap V^{\perp}$ have equal areas, equal perimeters, and equal radii of the smallest containing disks.

\medskip
A further geometrical application of the results of Theorem  \ref{th : main - 02} is the following corollary.

\begin{corollary}
Let $a\geq 0$ be an integer, $n=2^a$, and $X\subseteq \R^{2n}$ a finite set of points. 
There exist two projections $P\colon \R^{2n}\longrightarrow \R^{2n}$ and $Q\colon \R^{2n}\longrightarrow \R^{2n}$ of the Euclidean space $\R^{2n}$ onto mutually orthogonal $n$-dimensional subspaces of $\mathbb R^{2n}$ so that the two inertia tensors 
\[
I_P = \sum_{x\in X} Px\otimes Px\qquad\text{and}\qquad I_Q = \sum_{x\in X} Qx\otimes Qx
\]
are similar, that is transformable one to another by an orthogonal transformation.
\end{corollary}
\begin{proof}
The configuration space of all pairs of projectors $(P,Q)$ equipped with the $\Z/2$-action given by $(P,Q)\longmapsto (Q,P)$ can be identified with the Grassmann manifold $G_n(\mathbb R^{2n})$ where the corresponding $\Z/2$-action is $V\longmapsto V^{\perp}$. 

The tensors (the quadratic form) $I_P$ and $I_Q$ can be considered as a matrix which has the characteristic polynomials $\det (I_P - \lambda I_{2n})$ and $\det (I_Q - \lambda I_{2n})$ having at least $n$ zero roots (here $I_{2n}$ denotes the unit $2n\times 2n$ matrix).
Indeed, $I_P$ and $I_Q$, seen as matrices, have the kernels containing the $n$-dimensional subspaces $\ker P$ and $\ker Q$, respectively. 
Since every quadratic form can be taken to a diagonal quadratic form by an orthogonal transformation $I_P\longmapsto O^T I_P O$ without changing the characteristic polynomial, in order to show $I_P$ and $I_Q$ are transformable one to another it is sufficient to show that their characteristic polynomials are equal.

In both characteristic polynomials we only have $n$ nonzero coefficients at $\lambda^{2n-1},\ldots, \lambda^n$, call them $a(P)_{2n-1}, \ldots, a(P)_n$ for $I_P$ and $a(Q)_{2n-1}, \ldots, a(Q)_n$ for $I_Q$; they depend continuously on the pair $(P,Q)$. 
Since we want to find a pair  $(P,Q)$ with the property that associated characteristic polynomials coincide we consider the following $\Z/2$-map
\[
G_n(\mathbb R^{2n}) \longrightarrow \R^n,\quad (P,Q)\longmapsto \left( a(P)_{2n-1} - a(Q)_{2n-1},\ldots, a(P)_n - a(Q)_n \right).
\]
Like in the previous corollary, from the Grassmannian index estimate, this map must have a zero, which completes the proof.
\end{proof}

\medskip
This paper builds on the ideas and methods presented in \cite{Karasev2010} and presents the journal version of that preprint.

\medskip
\noindent
{\bf Acknowledgments} We thank Oleg Musin for drawing our attention to the problem of understanding the index of the orthogonal complement action on Grassmann manifolds.
We are grateful to the referees for helpful comments and suggestions.

\section{First estimates of the index}
\label{sec : estimates of the index}

In this section we make initial observations about the $\Z/2$ actions on the Grassmann manifolds $G_{n}(\R^{2n})$ and $\widetilde{G}_{n}(\R^{2n})$ and obtain some estimates of the corresponding indices.
Furthermore, we discuss the action of $\Z/2$ on the cohomology of the Grassmann manifold $G_{n}(\R^{2n})$.

\medskip
The first elementary observation we make is that the covering map between the Grassmann manifolds is a $\Z/2$-map.

\begin{proposition}
	\label{prop : covering map a Z/2-map}
	Let $n\geq 1$ be an integer.
	The covering map
	$
	c\colon \widetilde{G}_{n}(\R^{2n})\longrightarrow G_{n}(\R^{2n})
	$
	is a $\Z/2$-map, and consequently
	\[
	\ind_{\Z/2}(\widetilde{G}_{n}(\R^{2n});\F_2)\supseteq \ind_{\Z/2}(G_{n}(\R^{2n});\F_2).
	\]
\end{proposition}

\noindent
The covering map $c$ is a map that forgets orientation and is defined in general $c\colon \widetilde{G}_{n}(\R^{n+k})\longrightarrow G_{n}(\R^{n+k})$ for any $n\geq 1$ and $k\geq 1$, including the case when $n+k=\infty$.

\medskip
Next we construct $\Z/2$-maps between Grassmann manifolds of different dimensions.

\begin{proposition}
	\label{prop : Z/2-map between Grassmannians}
	Let $m\geq 1$ be an integer, and let $n$ be a multiple of $m$.
	There exist $\Z/2$-maps
	\[
	g\colon G_{m}(\R^{2m})\longrightarrow G_{n}(\R^{2n})
	\qquad\text{and}\qquad
	\widetilde{g}\colon \widetilde{G}_{m}(\R^{2m})\longrightarrow \widetilde{G}_{n}(\R^{2n}),
	\]
	and consequently
	\[
	\ind_{\Z/2}(G_{m}(\R^{2m});\F_2)\supseteq \ind_{\Z/2}(G_{n}(\R^{2n});\F_2)
	, \quad
	\ind_{\Z/2}(\widetilde{G}_{m}(\R^{2m});\F_2)\supseteq \ind_{\Z/2}(\widetilde{G}_{n}(\R^{2n});\F_2).
	\]
\end{proposition}
\begin{proof}
Let $d:=\tfrac{n}{m}$.
Choose a decomposition of the Euclidean space $\R^{2n}=(\R^{2m})^{\oplus d}$.
Then define the map $g$ and $\widetilde{g}$ by
\begin{equation}
	\label{eq : def of g}
	 \big( V\in G_{m}(\R^{2m})\big) \ \longmapsto \  \big(V^{\oplus d}\in G_{n}(\R^{2n})\big).
\end{equation}
In the case of the map $\widetilde{g}$ appropriate orientations of the spaces $V$ and $V^{\oplus d}$ are assumed.
By direct inspection the maps $g$ and $\widetilde{g}$ are $\Z/2$-maps.
\end{proof}

The $\Z/2$-map $g$ we just defined has an additional property.
Let $\gamma^n(\R^{n+k})$ denotes the canonical $n$-dimensional vector bundle over the Grassmann manifold $G_{n}(\R^{n+k})$, consult \cite[Lem.\,5.2]{Milnor1974}.

\begin{corollary}
\label{cor : Z/2-map between Grassmannians covered}
The $\Z/2$-map $g\colon G_{m}(\R^{2m})\longrightarrow G_{n}(\R^{2n})$ defined by \eqref{eq : def of g} is covered by a vector bundle map $\gamma^m(\R^{2m})^{\oplus d}\longrightarrow \gamma^n(\R^{2n})$.
\end{corollary}
\begin{proof}
	The vector bundle map that covers $g$ is defined by $(V;v_1,\ldots,v_d)\longmapsto (V^{\oplus d};v_1+\cdots+v_d)$.
\end{proof}

\medskip
Now we define a $\Z/2$-map from a Grassmann manifold into a sphere equipped with the antipodal action.

\begin{proposition}
	\label{prop : Z/2-map to sphere}
	Let $n\geq 1$ be an integer, and let the sphere $S^{2n-1}$ be equipped with the antipodal action.
	There exists a $\Z/2$-map
	$
	h\colon G_{n}(\R^{2n})\longrightarrow S^{2n-1}
	$,
	and consequently
	\[
	\ind_{\Z/2}G_{n}(\R^{2n})\supseteq \ind_{\Z/2}S^{2n-1}=\langle t^{2n} \rangle .
	\]
\end{proposition}
\begin{proof}
	For the definition of the map $h$ we consider elements $V$ of the Grassmann manifold $G_{n}(\R^{2n})$ as $(2n)\times (2n)$-matrices $A=(a_{i,j})$ which represent the orthogonal projection onto $V$.
	Such matrices fulfill the following properties
	\[
	A^{\top}=A, \quad A^2=A, \quad \mathrm{tr}A=n .
	\]
	Here $A^{\top}$ denotes the transpose matrix of $A$.
	The $\Z/2$-action on $G_{n}(\R^{2n})$ in the language of matrices is given by $\omega\cdot A=I-A$, where $I$ denotes the unit $(2n)\times (2n)$-matrix.
	First, we define the map $k\colon G_{n}(\R^{2n})\longrightarrow \R^{2n}{\setminus}\{0\}$ by
	\[
	k(A):=(a_{1,1}-\tfrac12,\,a_{1,2},\,a_{1,3},\,\ldots,\,a_{1,2n}).
	\]
	The map $k$ is well defined because $0\notin\mathrm{im}(k)$.
	Indeed, if $k(A)=0$ for some matrix $A$ then $\tfrac12$ is an eigenvalue of $A$ with eigenvector $(1,0,\cdots,0)$.
	This yields a contradiction with the requirement that $A^2=A$.
	Further on, if the action on $\R^{2n}{\setminus}\{0\}$ is assumed to be antipodal then $k$ is a $\Z/2$-map; as it can be demonstrated by the following commutative diagram:
\[
\xymatrix{
(a_{i,j})_{i,j=1,2n}\ar@{|->}[r]^-{k}\ar@{|->}[d]_{\omega\cdot}  &  (a_{1,1}-\tfrac12,a_{1,2},a_{1,3},\ldots,a_{1,2n})\ar@{|->}[d]_{\omega\cdot}\\
(\delta_{i,j}-a_{i,j})_{i,j=1,2n}\ar@{|->}[r]^-{k}  & (\tfrac12-a_{1,1},-a_{1,2},-a_{1,3},\ldots,-a_{1,2n}),
}
\]
where $\delta_{i,j}$ is the Kronecker symbol.
	Finally, the map $h$ is defined as the composition of $k$ and the radial retraction $\R^{2n}{\setminus}\{0\}\longrightarrow S^{2n-1}$ which is also a $\Z/2$-map.
\end{proof}

\medskip
The cohomology of the Grassmann manifold is described as a quotient polynomial algebra by the following classical result of Borel \cite[p.\,190]{Borel1953}.

\begin{theorem}
	\label{th : cohomology of Grassmaniann}
	Let $n,k\geq 1$ be integers.
	Then
	\[
	H^*(G_{n}(\R^{n+k});\F_2)\cong\F_2[w_1,\ldots,w_n;\bar{w}_1,\ldots,\bar{w}_k]/I_{n,k},
	\]
	where $\deg(w_i)=i$ for $1\leq i\leq n$, $\deg(\bar{w}_j)=j$ for $1\leq j\leq k$, and  the ideal $I_{n,k}$ is generated by $n+k$ relations which are derived from the equality
	\[
		(1+w_1+\cdots+w_n)(1+\bar{w}_1+\cdots+\bar{w}_k)=1.
	\]
\end{theorem}

Here the generators $w_i$, for $1\leq i\leq n$, can be identified with the Stiefel--Whitney classes of the canonical bundle $\gamma^n(\R^{n+k})$.
The remaining generators $\bar{w}_i$, for $1\leq i\leq k$, can be identified with the dual Stiefel--Whitney classes of $\gamma^n(\R^{n+k})$.

In the case when $n=k$ the $\Z/2$ action on the Grassmann manifold $G_{n}(\R^{2n})$ we consider, induces an action on the the corresponding cohomology $H^*(G_{n}(\R^{n+k});\F_2)$.
Let $\nu^n(\R^{2n})$ denote the normal bundle of the canonical vector bundle $\gamma^n(\R^{2n})$.
The continuous map $\omega\colon G_{n}(\R^{2n})\longrightarrow G_{n}(\R^{2n})$,  $V\longmapsto V^{\perp}$, is covered by a bundle map $\gamma^n(\R^{2n})\longrightarrow \nu^n(\R^{2n})$ given by $(V;v)\longmapsto (V^{\perp};v)$.
Since $w_i(\nu^n(\R^{2n}))=\bar{w}_i(\gamma^n(\R^{2n}))$ the naturality on the Stiefel--Whitney classes \cite[Ax.\,2,\,p.\,35]{Milnor1974} implies that $\omega$ acts of the generators of the cohomology $H^*(G_{n}(\R^{2n});\F_2)$ as follows
\[
\omega\cdot w_i = \bar{w}_i
\qquad\qquad\text{and}\qquad\qquad
\omega\cdot \bar{w}_i=w_i,
\]
for all $1\leq i\leq n$.

Using the description of the $\Z/2$ action on the cohomology of the Grassmann manifold $G_{n}(\R^{2n})$ we compute Fadell--Husseini indices in the case when $n=1$ and $n=2$.

\begin{proposition}
	\label{prop : index for n=1 and n=2}
	~~
	\begin{compactenum}[\rm (i)]
	\item $\ind_{\Z/2}(G_{1}(\R^{2});\F_2)=\langle t^{2}\rangle$,	
	\item $\ind_{\Z/2}(G_{2}(\R^{4});\F_2)=\langle t^{4}\rangle$.
	\end{compactenum}
\end{proposition}
\begin{proof}
{\rm (i)}
When $n=1$ the homeomorphisms $G_{1}(\R^{2})\cong\RP^1\cong S^1$ allow us to identify $G_{1}(\R^{2})$ with a sphere $S^1$ equipped with a free $\Z/2$ action. Consequently,
\[
\ind_{\Z/2}(G_{1}(\R^{2});\F_2)=\ind_{\Z/2}(S^1;\F_2)=\langle t^{2}\rangle.
\]	
{\rm (ii)} In the case of the Grassmann manifold $G_{2}(\R^{4})$ we consider the Borel construction fibration
\[
\xymatrix{
G_{2}(\R^{4})\ar[r] & \E(\Z/2)\times_{\Z/2}G_{2}(\R^{4})\ar[r] & \B(\Z/2),
}
\]
and its associated Serre spectral sequence whose $E_2$-term is given by
\begin{equation}
	\label{eq : spectral sequence - 01}
	E_2^{i,j}=H^i(\B(\Z/2); \mathcal{H}^j(G_{2}(\R^{4});\F_2))\cong H^i(\Z/2; H^j(G_{2}(\R^{4});\F_2)).
\end{equation}
 
\medskip
The action of $\pi_1(\B(\Z/2))\cong\Z/2$ on the cohomology of the fiber $H^*(G_{2}(\R^{4});\F_2)$ is non-trivial, and thus we first describe this action.
From Theorem \ref{th : cohomology of Grassmaniann} follows that
\[
H^*(G_{2}(\R^{4});\F_2)\cong\F_2[w_1,w_2;\bar{w}_1,\bar{w}_2]/I_{2,2},
\]
where the ideal $I_{2,2}$ is given by
\[
I_{2,2}=\langle w_1+ \bar{w}_1,  w_2+ \bar{w}_2 + w_1\bar{w}_1 , w_1\bar{w}_2 +  w_2\bar{w}_1,w_2\bar{w}_2 \rangle.
\]
Thus, the cohomology $H^*(G_{2}(\R^{4});\F_2)$, as a $\Z/2$-module, can be presented as
\[
H^j(G_{2}(\R^{4});\F_2)=
\begin{cases}
	\langle 1\rangle\cong \F_2,  & j=0,\\
	\langle w_1\rangle\cong \langle \bar{w}_1\rangle\cong\F_2,  & j=1,\\
	\langle w_2,\bar{w}_2\rangle\cong\F_2[\Z/2],  & j=2,\\
	\langle w_1w_2\rangle\cong \langle w_1\bar{w}_2\rangle\cong\langle w_2\bar{w}_1\rangle\cong
	\langle \bar{w}_1\bar{w}_2\rangle\cong\F_2,& j=3,\\
	\langle w_2^2\rangle\cong \langle \bar{w}_2^2\rangle\ \cong\F_2,& j=4,\\
	0, & \text{otherwise}.
\end{cases}
\]
Consequently the $E_2$-term of the spectral sequence \eqref{eq : spectral sequence - 01} can be evaluated as follows:
\[
E^{i,j}_2(G_{2}(\R^{4}))\cong H^i(\Z/2; H^j(G_{2}(\R^{4});\F_2))\cong
\begin{cases}
	H^i(\Z/2; \F_2),  & j=0,1,3,4,\\
	H^i(\Z/2; \F_2[\Z/2]),  & j=2,\\
	0, & \text{otherwise}.
\end{cases}
\]

For an illustration of the $E_2$-term see Figure \ref{fig 1: E_2 term}.
Each row of the spectral sequence \eqref{eq : spectral sequence - 01} is a $H^*(\Z/2;\F_2)$-module.
In particular, the rows $0,1,3,4$ of the  $E_2$-term are free $H^*(\Z/2;\F_2)$-modules generated, respectively, by
\[
1\in E^{0,0}_2(G_{2}(\R^{4})), \ w_1\in E^{0,1}_2(G_{2}(\R^{4})), \ w_1w_2\in E^{0,3}_2(G_{2}(\R^{4})), \ w_2^2\in E^{0,4}_2(G_{2}(\R^{4})).
\]
In the first column of the second row we have that $E_2^{0,2}\cong H^0(\Z/2; \F_2[\Z/2])\cong \F_2[\Z/2]^{\Z/2}\cong \F_2\cong  \langle w_2 +\bar{w}_2\rangle$.
Further on, since $\F_2[\Z/2]$ is a free $\Z/2$-module the rest of the second row has to vanish, that is  $E_2^{i,2}\cong H^i(\Z/2; \F_2[\Z/2])=0$ for $i\geq 1$.

The module structure on the second row $E^{*,2}_2(G_{2}(\R^{4}))$ is obvious and is given by $t\cdot (w_2+\bar{w}_2)=0$.

\begin{figure}
\centering
\includegraphics[scale=1]{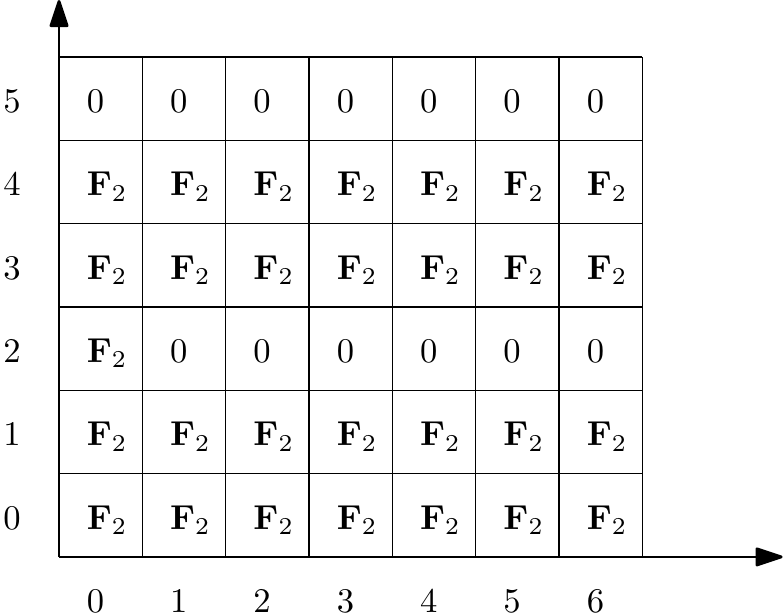}
\caption{\small The $E_2$-term of the spectral sequence \eqref{eq : spectral sequence - 01}. }
\label{fig 1: E_2 term}
\end{figure}

\medskip
First we prove that the second differential $\partial_2$ vanishes on $w_1\in E^{0,1}_2$, that is $\partial_2(w_1)=0$.
For that we use a comparison of spectral sequences.
Since $2$ is a multiple of $1$, according to Proposition \ref{prop : Z/2-map between Grassmannians}, there exists a $\Z/2$-map $g\colon G_1(\R^2)\longrightarrow G_2(\R^4)$.
The map $g$ induces a bundle morphism between the following Borel construction fibrations:
\[
\xymatrix{
\E(\Z/2)\times_{\Z/2}G_1(\R^2)\ar[rr]^-{\id\times_{\Z/2}g}\ar[d] & &\E(\Z/2)\times_{\Z/2}G_2(\R^4)\ar[d]\\
\B(\Z/2)\ar[rr]^-{\id} & &\B(\Z/2).
}
\]
The bundle morphism furthermore induces a morphism between corresponding Serre spectral sequences
\[
E^{i,j}_k(g)\colon E^{i,j}_k(G_{1}(\R^{2}))\longleftarrow E^{i,j}_k(G_{2}(\R^{4})),
\qquad\qquad k\geq 2
\]
with the property that $E^{i,0}_2(g)=\id$ for all $i\in\Z$.

On the other hand Corollary \ref{cor : Z/2-map between Grassmannians covered} implies that the map $g$ is covered by a vector bundle map
$\gamma^1(\R^{2})^{\oplus 2}\longrightarrow \gamma^2(\R^{4})$.
Hence, from naturality of the Stiefel--Whitney classes we have that
\[
g^*(w_1)=g^*(w_1(\gamma^2(\R^{4})))=w_1(\gamma^1(\R^{2})^{\oplus 2})={2\choose 1}w_1(\gamma^1(\R^{2}))={2\choose 1}w_1'=0,
\]
where $w_1':=w_1(\gamma^1(\R^{2}))$, is the generator of $H^1(G_1(\R^{2});\F_2)\cong\F_2$.

Since both rows $E^{*,1}_2(G_{1}(\R^{2}))$ and $E^{*,1}_2(G_{2}(\R^{4}))$ are free $H^*(\Z/2;\F_2)$-modules generated respectively by $w_1$ and $w_1'$ it follows that the morphism $E^{i,1}_2(g)$ is the zero morphism.
The morphism of spectral sequences commutes with differentials and consequently,
\begin{multline*}
		\partial_2(w_1)=(\id\circ \partial_2)(w_1)= (E^{2,0}_2(g)\circ \partial_2)(w_1)= (\partial_2\circ E^{0,1}_2(g))(w_1)= \\
		\partial_2 (E^{0,1}_2(g)(w_1))=\partial_2 (g^*(w_1))=\partial_2 (0)=0.
\end{multline*}
Here we use the facts
\begin{eqnarray*}
E^{0,1}_2(G_{1}(\R^{2}))&=&H^0(\Z/2;H^1(G_{1}(\R^{2});\F_2))\cong H^1(G_{1}(\R^{2});\F_2)^{\Z/2}\cong H^1(G_{1}(\R^{2});\F_2),\\
E^{0,1}_2(G_{2}(\R^{4}))&=&H^0(\Z/2;H^1(G_{2}(\R^{4});\F_2))\cong H^1(G_{2}(\R^{4});\F_2)^{\Z/2}\cong H^1(G_{2}(\R^{4});\F_2),	
\end{eqnarray*}
that imply equality of the maps $E^{2,0}_2(g)=g^*$.
Furthermore, since differentials are $H^*(\Z/2;\F_2)$-module morphisms, $w_1$ is a generator of the first row of the spectral sequence \eqref{eq : spectral sequence - 01}, as a $H^*(\Z/2;\F_2)$-module, we conclude that the differential $\partial_2$ vanishes on the complete first row.
Consequently,
\[
E^{i,0}_3(G_{2}(\R^{4}))\cong E^{i,0}_2(G_{2}(\R^{4}))\cong H^i(\Z/2;\F_2).
\]

\medskip
Next, we prove that $\partial_2(w_2+\bar{w}_2)=0$ and $\partial_3(w_2+\bar{w}_2)=0$.
Let us assume first that $\partial_2(w_2+\bar{w}_2)=t^2\cdot w_1\neq 0$.
Then we have a contradiction:
\[
0\neq t^3\cdot w_1= t\cdot (t^2\cdot w_1) = t\cdot \partial_2(w_2+\bar{w}_2)=\partial_2(t\cdot(w_2+\bar{w}_2))=\partial_2(0)=0.
\]
Thus, $\partial_2(w_2+\bar{w}_2)=0$ and consequently $E^{*,2}_3(G_{2}(\R^{4}))\cong E^{*,2}_2(G_{2}(\R^{4}))$.
Similarly, if $\partial_3(w_2+\bar{w}_2)= t^3\cdot 1\neq 0$ then we again get a contradiction:
\[
0\neq t^4\cdot 1= t\cdot (t^3\cdot 1) = t\cdot \partial_3(w_2+\bar{w}_2)=\partial_3(t\cdot(w_2+\bar{w}_2))=\partial_3(0)=0.
\]
Therefore,
\[
E^{*,0}_4(G_{2}(\R^{4}))\cong E^{*,0}_3(G_{2}(\R^{4}))\cong E^{*,0}_2(G_{2}(\R^{4}))\cong H^*(\Z/2;\F_2).
\]
Moreover, for $0\leq i\leq 3$ we have the following sequence of isomorphisms
\[
E^{i,0}_{\infty}(G_{2}(\R^{4}))\cong \cdots\cong E^{i,0}_2(G_{2}(\R^{4}))\cong H^i(\Z/2;\F_2) \neq 0.
\]

\medskip
Finally, from the previous sequence of isomorphisms and the equality \eqref{eq : index from ss} we conclude that
\[
	\ind_{\Z/2} (G_{2}(\R^{4});\F_2) \subseteq E_2^{\geq 4,0}(G_{2}(\R^{4}))\cong H^{\geq 4}(\Z/2;\F_2)=\langle t^4\rangle .
\]
On the other hand from Proposition \ref{prop : Z/2-map to sphere} we have that
$\ind_{\Z/2} (G_{2}(\R^{4});\F_2)\supseteq \langle t^4\rangle$.
Hence,
\[
\ind_{\Z/2} (G_{2}(\R^{4});\F_2)=\langle t^4\rangle,
\]
and the proof of the proposition is complete.
\end{proof}

Now using Proposition \ref{prop : Z/2-map between Grassmannians}, Corollary \ref{cor : Z/2-map between Grassmannians covered} and comparison of Serre spectral sequences we evaluate the index $\ind_{\Z/2}(G_{n}(\R^{2n});\F_2)$ for all odd integers $n$.

\begin{corollary}
\label{cor : index for odds}
Let $n\geq 1$ be an odd integer.
Then
\[
\ind_{\Z/2}(G_{n}(\R^{2n});\F_2)=\langle t^{2}\rangle.
\]
\end{corollary}
\begin{proof}
	Let $n\geq 1$ be an odd integer.
	Since $n$ is a multiple of $1$ from Proposition \ref{prop : Z/2-map between Grassmannians} we get a $\Z/2$-map $g\colon G_{1}(\R^{2})\longrightarrow G_{n}(\R^{2n})$.
	As in the proof of Proposition \ref{prop : index for n=1 and n=2} the map $g$ induces a bundle morphism between the corresponding Borel construction fibrations:
\[
\xymatrix{
\E(\Z/2)\times_{\Z/2}G_1(\R^2)\ar[rr]^-{\id\times_{\Z/2}g}\ar[d] & &\E(\Z/2)\times_{\Z/2}G_n(\R^{2n})\ar[d]\\
\B(\Z/2)\ar[rr]^-{\id} & &\B(\Z/2).
}
\]
This bundle morphism induces furthermore a morphism between Serre spectral sequences
\[
E^{i,j}_k(g)\colon E^{i,j}_k(G_{1}(\R^{2}))\longleftarrow E^{i,j}_k(G_{n}(\R^{2n})),
\qquad\qquad k\geq 2
\]
such that $E^{i,0}_2(g)=\id$ for all $i\in\Z$.

The spectral sequence $E^{*,*}_*(G_{1}(\R^{2}))$ can be described completely.
First, the $E_2$-term is given by
\[
E_2^{i,j}(G_{1}(\R^{2}))=H^i(\B (\Z/2);\mathcal{H}^j(G_{1}(\R^{2});\F_2))=H^i(\Z/2;H^j(S^1;\F_2))\cong H^i(\Z/2;\F_2)\otimes H^j(S^1;\F_2).
\]
Since $G_{1}(\R^{2})\cong S^1$ is a free $\Z/2$-space there is a homotopy equivalence
\[
\E(\Z/2)\times_{\Z/2}G_1(\R^2)\simeq G_{1}(\R^{2})/(\Z/2)\cong\RP^1.
\]
The spectral sequence $E_*^{*,*}(G_{1}(\R^{2}))$ converges to the cohomology of the Borel construction:
\[
H^*(\E(\Z/2)\times_{\Z/2}G_1(\R^2);\F_2)\cong H^*(G_{1}(\R^{2})/(\Z/2);\F_2)\cong H^*(S^1/(\Z/2);\F_2).
\]
Consequently $E_{\infty}^{i,j}(G_{1}(\R^{2}))=0$ for $(i,j)\notin\{(0,0),(1,0)\}$.
Therefore $\partial_2(w_1')=t^2$, where $w_1':=w_1(\gamma^1(\R^{2}))$ is the generator of $H^1(G_1(\R^{2});\F_2)\cong\F_2$.	

Next, from Corollary \ref{cor : Z/2-map between Grassmannians covered} we know that $g$ is covered by a vector bundle map
$\gamma^1(\R^{2})^{\oplus n}\longrightarrow \gamma^n(\R^{2n})$.
The naturality of Stiefel--Whitney classes implies that
\[
g^*(w_1)=g^*(w_1(\gamma^n(\R^{2n})))=w_1(\gamma^1(\R^{2})^{\oplus n})={n\choose 1}w_1(\gamma^1(\R^{2}))=w_1'\neq 0.
\]

The rows of the spectral sequences $E^{*,1}_2(G_{1}(\R^{2}))$ and $E^{*,1}_2(G_{n}(\R^{2n}))$ are free $H^*(\Z/2;\F_2)$-modules generated respectively by $w_1'$ and $w_1$.
Thus the the morphism $E^{i,1}_2(g)$ is an isomorphism.
Since the morphism $E^{i,j}_k(g)$ of spectral sequences commutes with differentials we have that,
\begin{multline*}
		\partial_2(w_1)=(\id\circ \partial_2)(w_1)= (E^{2,0}_2(g)\circ \partial_2)(w_1)= (\partial_2\circ E^{0,1}_2(g))(w_1)= \hfill \\
		\partial_2 (E^{0,1}_2(g)(w_1))=\partial_2 (g^*(w_1))=\partial_2 (w_1')=t^2.
\end{multline*}
Hence the equality \eqref{eq : index from ss} implies the final conclusion:
\[
\ind_{\Z/2}(G_{n}(\R^{2n});\F_2)=\langle t^{2}\rangle.
\]
\end{proof}
\section{Stiefel--Whitney classes of the wreath product vector bundle}
\label{sec : Stiefel--Whitney classes of the wreath product vector bundle}

In this section we introduce a notion of the wreath square of a vector bundle and derive a formula for computation of its total Stiefel--Whitney class.
This formula will be essentially used in the proof of Theorem \ref{th : main - 01}.
We work in generality necessary for giving this proof.

\subsection{Wreath squares}
\label{sub : Wreath squares}
Let $B$ be a CW-complex, not necessarily finite.
The Cartesian square $B\times B$ is equipped with a $\Z/2$ action given by  $\omega\cdot (x,y):=(y,x)$, where $\omega$ is the generator of $\Z/2$ and $(x,y)\in B\times B$.
Then the product space $(B\times B)\times \E\Z/2$ is a free $\Z/2$-space if the diagonal action is assumed.
This means that $\omega\cdot (x,y,e):=(y,x,\omega\cdot e)$ for $(x,y,e)\in (B\times B)\times \E\Z/2$.
Consequently, the projection map
\[
p_1\colon (B\times B)\times \E\Z/2\longrightarrow (B\times B),\qquad (x,y,e)\longmapsto (x,y)
\]
is a $\Z/2$-map.
The quotient space $((B\times B)\times \E\Z/2)/(\Z/2)$ will be denoted either by $B\wr\Z/2$ or by $(B\times B)\times_{\Z/2} \E\Z/2$, and will be called the {\bf wreath square of the space} $B$.

Next, let $f\colon B_1\longrightarrow B_2$ be a continuous map.
Then the map
\[
(f\times f)\times\id \colon (B_1\times B_1)\times\E\Z/2\longrightarrow (B_2\times B_2)\times\E\Z/2
\]
is a $\Z/2$-map and consequently induces a continuous map between the wreath squares:
\[
 (f\times f)\times_{\Z/2}\id \colon (B_1\times B_1)\times_{\Z/2}\E\Z/2\longrightarrow (B_2\times B_2)\times_{\Z/2}\E\Z/2.
\]
The map $(f\times f)\times_{\Z/2}\id$ is alternatively denoted by $f\wr\Z/2$.

\medskip
Let $\xi:=(p_{\xi}\colon E(\xi)\longrightarrow B(\xi))$ be a real $n$-dimensional vector bundle over the CW-complex  $B(\xi)$ whose fiber is $F(\xi)$.
The total space of $\xi$ is $E(\xi)$, and $p_{\xi}$ is the corresponding projection map.
We consider the following pull-back $p_1^*(\xi\times\xi)$ of the product bundle vector $\xi\times\xi$:
\[
\xymatrix{
E(p_1^*(\xi\times\xi))\ar[rr]\ar[d] &  &  E(\xi\times\xi)\ar[d] \\
(B(\xi)\times B(\xi))\times \E\Z/2\ar[rr]^-{p_1} &  & B(\xi\times\xi).
}
\]
Here we recall that by definition $E(\xi\times\xi)=E(\xi)\times E(\xi)$ and $B(\xi\times\xi)=B(\xi)\times B(\xi)$, consult for example \cite[p.\,27]{Milnor1974}.
Unlike the product bundle $\xi\times\xi$ the pull-back bundle $p_1^*(\xi\times\xi)$ is equipped with a free $\Z/2$-action.
Moreover, the projection $E(p_1^*(\xi\times\xi))\longrightarrow (B(\xi)\times B(\xi))\times \E\Z/2$ is a $\Z/2$-map between free $\Z/2$-spaces.
Consequently, after taking quotient we get a $2n$-dimensional vector bundle
\[
E(p_1^*(\xi\times\xi))/(\Z/2)\longrightarrow (B(\xi)\times B(\xi))\times_{\Z/2} \E\Z/2
\]
that we denote by $\xi\wr\Z/2$ and call the {\bf wreath square of the vector bundle} $\xi$.

\medskip
The wreath square of the vector bundle behaves naturally with respect to the Whitney sum of vector bundles.
More precisely, if $\xi$ and $\eta$ are vector bundles over the same base space $B$ then
\begin{equation}
	\label{eq : Whitney sum of wreath square}
	(\xi\oplus\eta)\wr \Z/2\cong (\xi\wr\Z/2)\oplus(\eta\wr\Z/2).
\end{equation}
Indeed, for proving \eqref{eq : Whitney sum of wreath square} it suffices to exhibit a fiberwise isomorphism between the corresponding vector bundles, see \cite[Lem.\,2.3]{Milnor1974}.
The fibers of these bundles, respectively, are
\[
(F(\xi)\oplus F(\eta)) \times (F(\xi)\oplus F(\eta))
\qquad\text{and}\qquad
(F(\xi)\times F(\xi))\oplus (F(\eta)\times F(\eta)).
\]
Thus the obvious shuffling linear isomorphism gives the required isomorphism in \eqref{eq : Whitney sum of wreath square}.
\subsection{Cohomology of the wreath square of a space}
\label{sub : Cohomology of the wreath square of a space}
Let $B$ be a CW-complex and $(B\times B)\times_{\Z/2} \E\Z/2$ its wreath square.
The wreath square of $B$ is the total space of the following fiber bundle
\begin{equation}
	\label{eq : fib-001}
	\xymatrix{
B\times B\ar[r] & (B\times B)\times_{\Z/2}\E\Z/2\ar[r] & \B (\Z/2).
}
\end{equation}
The Serre spectral sequence associated to the fibration \eqref{eq : fib-001} has $E_2$-term given by
\begin{equation}
	\label{eq : fib-005}
	E_{2}^{i,j}:=H^i(\B(\Z/2);\mathcal{H}^j(B\times B;\F_2))\cong H^i(\Z/2;H^j(B\times B;\F_2)).
\end{equation}
The spectral sequence $\eqref{eq : fib-005}$ converges to the cohomology of the wreath square $H^*((B\times B)\times_{\Z/2}\E\Z/2;\F_2)$.
Here the local coefficient system $\mathcal{H}^*(B\times B;\F_2)$ is determined by the action of the fundamental group of the base $\pi_1(\B (\Z/2))\cong\Z/2$ on the cohomology of the fiber $H^*(B\times B;\F_2)$.
The K\"unneth formula gives a presentation of the cohomology of the fiber in the form
\[
H^*(B\times B;\F_2)\cong H^*(B;\F_2)\otimes H^*(B;\F_2),
\]
and the action of $\Z/2$ interchanges factors in the tensor product.

The $E_2$-term of the spectral sequence \eqref{eq : fib-005} can be described in more details using \cite[Cor\,IV.1.6]{Adem2004}.

\begin{proposition}
\label{lem : 02}
Let $\mathcal{B}:=\{v_k : k\in K\}$ be a basis of the $\F_2$ vector space $H^*(B;\F_2)$, where the index set $K$ is equipped with a linear order.
The $E_2$-term of the spectral sequence \eqref{eq : fib-005} can be presented as follows
\begin{eqnarray*}
	E_{2}^{i,j} &=& 	
	\begin{cases}
		H^j(B\otimes B;\F_2)^{\Z/2}, & i=0,\\
		H^{j/2}(B;\F_2),				 & i>0,\,j\text{ even},\\
		0,							 & \text{otherwise}.
	\end{cases}
\end{eqnarray*}
Furthermore, $E_{2}^{*,*}$ as a $H^*(\Z/2;\F_2)$-module decomposes into the direct sum
\[
\bigoplus_{k\in K} H^*(\Z/2;\F_2) \oplus \bigoplus_{k_1<k_2\in K} \F_2,
\]
where the action of $H^*(\Z/2;\F_2)$ on each summand of the first sum is given by the cup product, and on the each summand of the second sum is trivial.
\end{proposition}

From the classical work of Nakaoka \cite{Nakaoka1961}, see also \cite[Thm.\,IV.1.7]{Adem2004}, we have that the $E_2$-term of the spectral sequence \eqref{eq : fib-005} collapses.

\begin{proposition}
	\label{th : E_2 collapses}
	Let $B$ be a CW-complex.
	The Serre spectral sequence of the fibration \eqref{eq : fib-001} collapses at the $E_2$-term, that means $E_{2}^{i,j}\cong E_{\infty}^{i,j}$ for all $(i,j)\in\Z\times\Z$.
\end{proposition}

In order to give a description of the $E_{\infty}$-term of the spectral sequence \eqref{eq : fib-005}, and therefore present the cohomology $H^*((B\times B)\times_{\Z/2}\E\Z/2;\F_2)$, we introduce the following maps.
First we consider the map (not a homomorphism)
\[
P\colon H^j(B;\F_2)\longrightarrow H^{2j}(B\times B;\F_2)^{\Z/2}\cong E^{0,2j}_2\cong E^{0,2j}_{\infty}
\]
given by $P(a):= a\otimes a$, for $a\in H^j(B;\F_2)$ and $j\in\Z$.
From direct computation we have that the map $P$ is not additive but is multiplicative.
Next we consider the map
\[
Q\colon H^{j}(B\times B;\F_2) \longrightarrow H^{j}(B\times B;\F_2)^{\Z/2}
\]
given by $Q(a\otimes b):= a\otimes b + b\otimes a$, for $a\otimes b\in H^{j}(B\times B;\F_2)$ and $j\in\Z$.
The map $Q$ is additive but not multiplicative.

Now from Propositions \ref{lem : 02} and \ref{th : E_2 collapses}, using maps $P$ and $Q$, we have that:
\begin{compactitem}
\item $E^{i,0}_2\cong E^{i,0}_{\infty}\cong H^i(\Z/2;\F_2)\cong\F_2[t]$,
\item $E^{0,j}_2 \cong E^{0,j}_{\infty} \cong  H^{j}(B\times B;\F_2)^{\Z/2}$,
\item $E^{0,j}_2\cong E^{0,j}_{\infty} \cong P(H^{j/2}(B;\F_2))\oplus Q(H^{j}(B\times B;\F_2))$ for even $j\geq 2$,
\item $E^{i,j}_2\cong E^{i,j}_{\infty}\cong P(H^{j/2}(B;\F_2))\otimes H^*(\Z/2;\F_2)$ for $j\geq 2$ even and $i\geq 1$,
\end{compactitem}
where $\deg(t)=1$.
From the previous description of the spectral sequence, multiplicative property of the map $P$, and the relation
\begin{equation}
	\label{eq : multiplication Q and t}
	Q(H^{j}(B\times B;\F_2))\cdot  t=0,
\end{equation}
we have a description of the multiplicative structure of the $E_{\infty}$-term as well as of the cohomology of the wreath square $H^*((B\times B)\times_{\Z/2}\E\Z/2;\F_2)$.
For the fact that no multiplicative extension problem arrises see for example \cite[Rem. after Thm.\,2.1]{Leary1997}.

\medskip
The next property of the cohomology of the wreath square will be essential in application of the Splitting principle.

\begin{proposition}
	\label{prop : 1-1 in cohomology of wreath product}
	Let $B_1$ and $B_2$ be CW-complexes, and let $f\colon B_1\longrightarrow B_2$ be a continuous map.
	If the induced map in cohomology $f^*\colon H^*(B_2;\F_2)\longrightarrow H^*(B_1;\F_2)$ is injective, then the induced map between cohomologies of wreath squares
\[
((f\times f)\times_{\Z/2}\id)^* \colon H^*((B_2\times B_2)\times_{\Z/2}\E\Z/2;\F_2)\longrightarrow H^*((B_1\times B_1)\times_{\Z/2}\E\Z/2;\F_2),
\]
or in different notation
\[
(f\wr\Z/2)^*\colon H^*(B_2\wr\Z/2;\F_2)\longrightarrow  H^*(B_1\wr\Z/2;\F_2),
\]
is injective.
\end{proposition}
\begin{proof}
	The induced map $f\wr\Z/2=(f\times f)\times_{\Z/2}\id$ between wreath squares defines the following morphism of fibrations
\[
\xymatrix{
 B_1\wr\Z/2\ar[rr]^-{f\wr\Z/2}\ar[d]  & &  B_1\wr\Z/2\ar[d]\\
\B (\Z/2)\ar[rr]^-{\id} & & \B (\Z/2).
}
\]
Consequently, it induces a morphism between the corresponding spectral sequences
\[
E^{i,j}_2(f\wr\Z/2)\colon
E^{i,j}_2( B_1\wr\Z/2)\longleftarrow E^{i,j}_2( B_2\wr\Z/2).
\]
From Proposition \ref{th : E_2 collapses} we have that $E^{i,j}_2(f\wr\Z/2)=E^{i,j}_{\infty}(f\wr\Z/2)$, and furthermore from Proposition \ref{lem : 02} that $E^{i,j}_2(f\wr\Z/2)$ is injective for every $i,j\in\Z$.
Since we are working over a field the absence of an additive extension problem implies that the map $(f\wr\Z/2)^*$ is also injective.
\end{proof}

\subsection{The total Stiefel--Whitney class of the wreath square of a vector bundle}
To every real $n$-dimensional vector bundle $\xi$ we associate the characteristic class
\[
u(\xi):=w(\xi\wr\Z/2)\in H^*(B(\xi)\wr \Z/2;\F_2),
\]
which is the total Stiefel--Whitney class of the $2n$-dimensional vector bundle $\xi\wr\Z/2$ living in the cohomology of the wreath square of the total space.
The assignement $\xi \longmapsto u(\xi)$ we just defined is natural with respect to continuous maps.
Indeed, for a continuous map $f\colon B\longrightarrow B(\xi)$ consider the following commutative diagram of vector bundle maps (fiberwise linear isomorphism):
\[
\xymatrix@C=1.9em{
& & E(p_1^*(\xi\times\xi))\ar[rrrr]\ar[dd]|!{[d]}\hole  & & & &  E(\xi\times\xi)\ar[dd]\\
E(q_1^*(f^*\xi\times f^*\xi) )\ar[urr]\ar[rrrr]\ar[dd] & &   & & E(f^*\xi\times f^*\xi)\ar[urr]\ar[dd]  & &\\
& & (B(\xi)\times B(\xi))\times\E\Z/2\ar[rrrr]|!{[rr]}\hole^{p_1} & & & &  B(\xi)\times B(\xi)\\
(B\times B)\times\E\Z/2\ar[urr]_{(f\times f)\times\id}\ar[rrrr]^{q_1} & &   & & B\times B,\ar[urr]_{f\times f}  & &
}
\]
where $q_1\colon (B\times B)\times\E\Z/2\longrightarrow B\times B$ is the projection.
After taking the quotient of the left hand side square of the diagram, with respect to the free $\Z/2$-action, we get the following vector bundle map between wreath product squares:
\[
\xymatrix{
E(q_1^*(f^*\xi\times f^*\xi) )/(\Z/2) \ar[rr]\ar[d] & &  E(p_1^*(\xi\times\xi))/(\Z/2)\ar[d]\\
B\wr\Z/2\ar[rr]^-{f\wr\Z/2} & & B(\xi)\wr\Z/2.
}
\]
In particular this means that $(f\wr\Z/2)^*(\xi\wr\Z/2)=f^*\xi\wr\Z/2 $.
Consequently,
\[
(f\wr\Z/2)^*(u(\xi))=u(f^*(\xi)).
\]

Now we are ready to give and verify the formula for computing the characteristic class $u(\xi)$.

\begin{theorem}
	\label{thm : SW classes of the wreath product}
	Let $\xi$ be a real $n$-dimensional vector bundle over a $CW$-complex.
	Then
\begin{equation}
\label{eq : the main formula}
u(\xi)	=\sum_{0\leq r< s\leq n}Q(w_r(\xi)\otimes w_s(\xi))+
\sum_{0\leq r\leq n} P(w_r(\xi))\cdot(1+t)^{n-r}
\end{equation}
\end{theorem}

The proof of Theorem \ref{thm : SW classes of the wreath product} will be conducted using a modification of the Splitting principle.
Thus we first give a proof in the special case of a line bundle.

\begin{proposition}
	\label{prop : the case of line bundle}
Let $\xi$ be a real $1$-dimensional vector bundle over a $CW$-complex.
	Then
\begin{multline*}
\label{eq : the main formula}
u(\xi)	= Q(w_0(\xi)\otimes w_1(\xi))+P(w_0(\xi))\cdot(1+t)+P(w_1(\xi))=\\
1+\big(t+ 1\otimes w_1(\xi) + w_1(\xi)\otimes 1 \big)+w_1(\xi)\otimes w_1(\xi).
\end{multline*}
\end{proposition}
\begin{proof}
First, consider the following maps of vector bundles:
\[
\xymatrix@1{
E(\xi\wr\Z/2)\ar@{=}[r] & E(p_1^*(\xi\times\xi))/(\Z/2)\ar[d] & & E(p_1^*(\xi\times\xi))\ar[ll]\ar[d]\ar[rr] & & E(\xi)\times E(\xi)\ar[d]\\
B(\xi)\wr\Z/2\ar@{=}[r] &(B(\xi)\times B(\xi))\times_{\Z/2} \E\Z/2 & &(B(\xi)\times B(\xi))\times \E\Z/2\ar[ll]_-{p_2}\ar[rr]^-{p_1} & & B(\xi)\times B(\xi).
}
\]
The naturality property of Stiefel--Whitney classes \cite[Ax.\,2, p.\,37]{Milnor1974} in combination with formula \cite[Prob.\,4-A, p.\,54]{Milnor1974} implies that
\begin{multline}
\label{rel-01}
p_2^*(u(\xi))=w(p_1^*(\xi\times\xi))=p_1^*(w(\xi\times\xi))=
p_1^*\big(1\times 1+(1\times w_1(\xi)+w_1(\xi)\times 1)+w_1(\xi)\times w_1(\xi)\big)	=\hfill\\
p_1^*\big(1\otimes 1+(1\otimes w_1(\xi)+w_1(\xi)\otimes 1)+w_1(\xi)\otimes w_1(\xi)\big),	
\end{multline}
where we silently use the Eilenberg--Zilber isomorphism \cite[Th.\,VI.3.2]{Bredon1997}.
Since $p_1$ is the projection we furthermore have that
\begin{multline*}
p_1^*\big(1\otimes 1+(1\otimes w_1(\xi)+w_1(\xi)\otimes 1)+w_1(\xi)\otimes w_1(\xi)\big)=\hfill\\
1\otimes 1\otimes 1+(1\otimes w_1(\xi)\otimes 1+w_1(\xi)\otimes 1\otimes 1)+w_1(\xi)\otimes w_1(\xi)\otimes 1,
\end{multline*}
and consequently
\[
p_2^*(u(\xi))=1\otimes 1\otimes 1+(1\otimes w_1(\xi)\otimes 1+w_1(\xi)\otimes 1\otimes 1)+w_1(\xi)\otimes w_1(\xi)\otimes 1.
\]

\smallskip
Next consider an arbitrary point $b\in\B(\xi)$ and the inclusion map $h\colon \{b\}\longrightarrow B(\xi)$ and  the induced map between wreath squares $h\wr\Z/2\colon \{b\}\wr\Z/2\longrightarrow \B(\xi)\wr\Z/2$.
Then we have a map of vector bundles going from $\xi|_{\{b\}}\wr\Z/2$ to $\xi\wr\Z/2$, that is
\[
\xymatrix{
E(\xi|_{\{b\}}\wr\Z/2)\ar[rr]\ar[d]& & E(\xi\wr\Z/2)\ar[d]\\
B(\xi|_{\{b\}}\wr\Z/2)\ar[rr]^-{h\wr\Z/2}& & B(\xi\wr\Z/2) ,
}
\]
where $B(\xi\wr\Z/2)=(B(\xi)\times B(\xi))\times_{\Z/2} \E\Z/2$ and $E(\xi\wr\Z/2)=E(p_1^*(\xi\times\xi))/(\Z/2)$.
Consequently,
\[
(h\wr\Z/2)^*(u(\xi))=w(\xi|_{\{b\}}\wr\Z/2).
\]
Since $B(\xi|_{\{b\}}\wr\Z/2)\cong(\E\Z/2)/(\Z/2)\cong\B (\Z/2)$ by direct inspection we see that $\xi|_{\{b\}}\wr\Z/2$ decomposes into a sum of a trivial line and non-trivial line vector bundle.
More precisely we can explicitly compute that $w(\xi|_{\{b\}}\wr\Z/2)=1+t$.
Therefore
\begin{equation}
\label{rel-02}
(h\wr\Z/2)^*(u(\xi))=1+t.	
\end{equation}

\smallskip
We have already discuss how the ambient cohomology $H^*(B(\xi)\wr \Z/2;\F_2)$ for the characteristic class $u(\xi)$ can be computed from the spectral sequence $\eqref{eq : fib-005}$.
In particular, we can interpret that
\[
u(\xi)\in E_2^{0,0}\oplus E_2^{1,0}\oplus E_2^{0,1}\oplus E_2^{2,0}\oplus E_2^{1,1}\oplus E_2^{0,2}\cong
E_{\infty}^{0,0}\oplus E_{\infty}^{1,0}\oplus E_{\infty}^{0,1}\oplus E_{\infty}^{2,0}\oplus E_{\infty}^{1,1}\oplus E_{\infty}^{0,2}.
\]
Now from relations \eqref{rel-01} and \eqref{rel-02}, and the fact that $E_2^{1,1}=0$, we get that
 \[
 u(\xi)=1+\big(t+ 1\otimes w_1(\xi) + w_1(\xi)\otimes 1 \big)+w_1(\xi)\otimes w_1(\xi),
 \]
 concluding the proof of the proposition.
\end{proof}

In order to prove the general case of Theorem \ref{thm : SW classes of the wreath product} we use the general idea of the Splitting principle.
For that we first introduce a notion of a splitting map of a vector bundle.

Let $\xi$ be a vector bundle over the base space $B(\xi)$.
A continuous map $f\colon B\longrightarrow B(\xi)$ from a space $B$ to the base space of $\xi$ is called a {\bf splitting map} if the pull-back bundle $f^*\xi$ is isomorphic to a Whitney sum of line bundles and the induced homomorphism $f^*\colon H^*(B(\xi);\F_2)\longrightarrow H^*(B;\F_2)$ is a monomorphism.
The key property that we use is the following classical fact, see for example \cite[Prop.\,17.5.2]{Husemoller1994}.

\begin{proposition}
	\label{prop : splitting map}
	For any vector bundle $\xi$ there exists a splitting map.
\end{proposition}

Now we utilize several facts that we established to give the proof of Theorem \ref{thm : SW classes of the wreath product}.

\begin{proof}[Proof of Theorem \ref{thm : SW classes of the wreath product}]
Let $\xi$ be a real $n$-dimensional vector bundle over a $CW$-complex $B(\xi)$.
Then according to Proposition \ref{prop : splitting map} there exists a splitting map $f\colon B\longrightarrow B(\xi)$.
Hence, $f^*\xi\cong \alpha_1\oplus\cdots\oplus\alpha_n$, where $\alpha_1,\ldots,\alpha_n$ are line bundles over $B$.
Moreover, $f^*\colon H^*(B(\xi);\F_2)\longrightarrow H^*(B;\F_2)$ is a monomorphism that sends, due to naturality and the Whitney sum formula,  Stiefel--Whitney classes of $\xi$ into elementary symmetric polynomials in ``variables'' $w_1(\alpha_1),\ldots,w_1(\alpha_n) $, that is
\[
w_k(f^*\xi)=\sigma_k(w_1(\alpha_1),\ldots, w_1(\alpha_n)), \qquad 1\leq k \leq n,
\]
where $\sigma_k(x_1,\ldots,x_n)=\sum_{1\leq i_1<i_2<\cdots<i_k\leq n}x_{i_1}\cdots x_{i_k}$ denotes the $k$-th elementary symmetric polynomials in variables $x_1,\ldots,x_n$.

Now, the isomorphism of vector bundles \eqref{eq : Whitney sum of wreath square} yields that
\[
(f^*\xi)\wr\Z/2\cong (\alpha_1\oplus\cdots\oplus\alpha_n)\wr\Z/2\cong (\alpha_1\wr\Z/2)\oplus\cdots\oplus (\alpha_n\wr\Z/2).
\]
Consequently,
\begin{multline*}
u(f^*\xi)=w((f^*\xi)\wr\Z/2)=w((\alpha_1\wr\Z/2)\oplus\cdots\oplus (\alpha_n\wr\Z/2))\hfill \\
=w(\alpha_1\wr\Z/2)\cdots w(\alpha_n\wr\Z/2)	=u(\alpha_1)\cdots u(\alpha_n).
\end{multline*}
Next, Proposition \ref{prop : the case of line bundle} implies that for all $1\leq k\leq n$:
\[
u(\alpha_k)=1+\big(t+ 1\otimes w_1(\alpha_k) + w_1(\alpha_k)\otimes 1 \big)+w_1(\alpha_k)\otimes w_1(\alpha_k).
\]
Combining last two equalities get that
\begin{multline}
\label{equality-00}
u(f^*\xi)=\prod_{k=1}^n\Big(1+\big(t+ 1\otimes w_1(\alpha_k) + w_1(\alpha_k)\otimes 1 \big)+w_1(\alpha_k)\otimes w_1(\alpha_k)\Big)=\\
\prod_{k=1}^n\Big(1+\big(t+ Q(1\otimes w_1(\alpha_k) )\big)+P(w_1(\alpha_k))\Big).
\end{multline}

Now we prove the equality
\begin{multline}
\label{equality-01}
\prod_{k=1}^n\Big(1+\big(t+ Q(1\otimes w_1(\alpha_k) )\big)+P(w_1(\alpha_k))\Big)=\\
\sum_{0\leq r< s\leq n}Q(w_r(f^*\xi)\otimes w_s(f^*\xi))+
\sum_{0\leq r\leq n} P(w_r(f^*\xi))\cdot(1+t)^{n-r}
\end{multline}
applying an induction on $n$.
It is important to recall that $w_k(f^*\xi)=\sigma_k(w_1(\alpha_1),\ldots, w_1(\alpha_n))$ for all $1\leq k \leq n$, meaning that we want to prove
\begin{multline}
\label{equality-012}
\prod_{k=1}^n\Big(1+\big(t+ Q(1\otimes w_1(\alpha_k) )\big)+P(w_1(\alpha_k))\Big)=\\
\sum_{0\leq r< s\leq n}Q(\sigma_r(w_1(\alpha_1),\ldots, w_1(\alpha_n))\otimes \sigma_s(w_1(\alpha_1),\ldots, w_1(\alpha_n)))+\\
\sum_{0\leq r\leq n} P(\sigma_r(w_1(\alpha_1),\ldots, w_1(\alpha_n)))\cdot(1+t)^{n-r}.
\end{multline}
The induction basis $n=1$ is holds since both sides of~\eqref{equality-012} are identical for $n=1$.
For clarity of the induction proof we first present the proof of case $n=2$.
having in mind that $P$ is multiplicative, $Q$ is additive and $t\cdot Q(\cdot)=0$ we have
{\small
\begin{multline}
\label{equality-02}
\Big(1+\big(t+ Q(1\otimes w_1(\alpha_1) )\big)+P(w_1(\alpha_1))\Big)\Big(1+\big(t+ Q(1\otimes w_1(\alpha_2) )\big)+P(w_1(\alpha_2))\Big)=\hfill \\
(1+t)^2+Q(1\otimes w_1(\alpha_2))+(1+t)P(w_1(\alpha_2))+Q(1\otimes w_1(\alpha_1))+Q(1\otimes w_1(\alpha_1))Q(1\otimes w_1(\alpha_2))+\\
Q(1\otimes w_1(\alpha_1))P(w_1(\alpha_2))+(1+t)P(w_1(\alpha_1))+P(w_1(\alpha_1))Q(1\otimes w_1(\alpha_2))+P(w_1(\alpha_1))P(w_1(\alpha_2)).
\end{multline}}
In order to make computation more transparent we separately evaluate different pieces of the last sum separately.
The first piece is:
{\small
\begin{multline*}
(1+t)\big(P(w_1(\alpha_1))+P(w_1(\alpha_2))\big)=
(1+t)\big(P(w_1(\alpha_1)+w_1(\alpha_2))+Q(w_1(\alpha_1)\otimes w_1(\alpha_2)) \big)=\\
(1+t)P(w_1(\alpha_1)+w_1(\alpha_2))+Q(w_1(\alpha_1)\otimes w_1(\alpha_2))=
(1+t)P(w_1(f^*\xi))+Q(w_1(\alpha_1)\otimes w_1(\alpha_2)).
\end{multline*}}
The next piece is:
{\small
\begin{multline*}
Q(1\otimes w_1(\alpha_1))Q(1\otimes w_1(\alpha_2))=
(1\otimes w_1(\alpha_1)+ w_1(\alpha_1)\otimes 1)(1\otimes w_1(\alpha_2)+ w_1(\alpha_2)\otimes 1)=\hfill\\
\qquad\qquad\qquad 1\otimes w_1(\alpha_1)w_1(\alpha_2)+w_1(\alpha_1)\otimes w_1(\alpha_2)+w_1(\alpha_2)\otimes w_1(\alpha_1)+
w_1(\alpha_1)w_1(\alpha_2)\otimes 1=\hfill\\
Q(1\otimes w_1(\alpha_1)w_1(\alpha_2))+Q(w_1(\alpha_1)\otimes w_1(\alpha_2))=
Q(1\otimes w_2(f^*\xi))+Q(w_1(\alpha_1)\otimes w_1(\alpha_2)).
\end{multline*}}
In the third piece we compute
{\small
\begin{multline*}
Q(1\otimes w_1(\alpha_1))P(w_1(\alpha_2))+ P(w_1(\alpha_1))Q(1\otimes w_1(\alpha_2))=\hfill\\
(1\otimes w_1(\alpha_1)+w_1(\alpha_1)\otimes 1)(w_1(\alpha_2)\otimes w_1(\alpha_2))+(w_1(\alpha_1)\otimes w_1(\alpha_1))(1\otimes w_1(\alpha_2)+w_1(\alpha_2)\otimes 1)=\\
w_1(\alpha_2)\otimes w_1(\alpha_1)w_1(\alpha_2)+w_1(\alpha_1)w_1(\alpha_2)\otimes w_1(\alpha_2)+
w_1(\alpha_1)\otimes w_1(\alpha_1)w_1(\alpha_2)+w_1(\alpha_1)w_1(\alpha_2)\otimes w_1(\alpha_1)=\\
(w_1(\alpha_1)+w_1(\alpha_2))\otimes w_1(\alpha_1)w_1(\alpha_2)+
w_1(\alpha_1)w_1(\alpha_2)\otimes (w_1(\alpha_1)+w_1(\alpha_2))=
Q(w_1(f^*\xi)\otimes w_2(f^*\xi)).
\end{multline*}}
The final piece yields:
{\small
\begin{multline*}
(1+t)^2+Q(1\otimes w_1(\alpha_2))+Q(1\otimes w_1(\alpha_1))+P(w_1(\alpha_1))P(w_1(\alpha_2))=\hfill\\
(1+t)^2+Q(1\otimes (w_1(\alpha_2)+w_1(\alpha_1)))+P(w_1(\alpha_1)w_1(\alpha_2))=
(1+t)^2+Q(1\otimes w_1(f^*\xi))+P(w_2(f^*\xi)).
\end{multline*}
}

\noindent
Gathering all these calculations together in \eqref{equality-02} we have that:
{\small
\begin{multline*}
\Big(1+\big(t+ Q(1\otimes w_1(\alpha_1) )\big)+P(w_1(\alpha_1))\Big)\Big(1+\big(t+ Q(1\otimes w_1(\alpha_2) )\big)+P(w_1(\alpha_2))\Big)=\hfill \\
\qquad
Q(1\otimes w_1(f^*\xi))+Q(1\otimes w_2(f^*\xi))+Q(w_1(f^*\xi)\otimes w_2(f^*\xi))+P(w_2(f^*\xi))+(1+t)P(w_1(f^*\xi))+(1+t)^2.
\end{multline*}}
This concludes the proof of the case $n=2$.
Now let us assume that the equalities \eqref{equality-01} and \eqref{equality-012} hold for all $k\leq n-1$.
Using the induction hypothesis we calculate as follows:
{\small
\begin{multline*}
\prod_{k=1}^n\Big(1+\big(t+ Q(1\otimes w_1(\alpha_k) )\big)+P(w_1(\alpha_k))\Big)= \hspace{70pt}\hfill\\
\Big(1+\big(t+ Q(1\otimes w_1(\alpha_1) )\big)+P(w_1(\alpha_1))\Big)\prod_{k=2}^n\Big(1+\big(t+ Q(1\otimes w_1(\alpha_k) )\big)+P(w_1(\alpha_k))\Big)=\hspace{90pt}\\
\Big(1+\big(t+ Q(1\otimes w_1(\alpha_1) )\big)+P(w_1(\alpha_1))\Big)
\Big(
\sum_{0\leq r< s\leq n-1}Q(\sigma_r(w_1(\alpha_2),\ldots, w_1(\alpha_n))\otimes \sigma_s(w_1(\alpha_2),\ldots, w_1(\alpha_n)))+ \\
\sum_{0\leq r\leq n-1} P(\sigma_r(w_1(\alpha_2),\ldots, w_1(\alpha_n)))\cdot(1+t)^{n-1-r}
\Big) .
\end{multline*}}
By a direct calculation we get that
{\small
\begin{multline*}
\prod_{k=1}^n\Big(1+\big(t+ Q(1\otimes w_1(\alpha_k) )\big)+P(w_1(\alpha_k))\Big)=\\
\sum_{0\leq r< s\leq n}Q(\sigma_r(w_1(\alpha_1),\ldots, w_1(\alpha_n))\otimes \sigma_s(w_1(\alpha_1),\ldots, w_1(\alpha_n)))+\\
\sum_{0\leq r\leq n} P(\sigma_r(w_1(\alpha_1),\ldots, w_1(\alpha_n)))\cdot(1+t)^{n-r},
\end{multline*}}
which concludes the induction proof.

Finally, from equalities \eqref{equality-00} and \eqref{equality-01} we have that
\[
u(f^*\xi)=
\sum_{0\leq r< s\leq n}Q(w_r(f^*\xi)\otimes w_s(f^*\xi))+
\sum_{0\leq r\leq n} P(w_r(f^*\xi))\cdot(1+t)^{n-r}.
\]
Furthermore, the naturally of the wreath product construction implies that
\begin{multline}
	\label{equality-004}
(f\wr\Z/2)^*(u(\xi))=u(f^*\xi)=\\
(f\wr\Z/2)^*\Big(
\sum_{0\leq r< s\leq n}Q(w_r(\xi)\otimes w_s(\xi))+
\sum_{0\leq r\leq n} P(w_r(\xi))\cdot(1+t)^{n-r}
\Big).
\end{multline}
Since $f$ is a splitting map the induced map in cohomology $f^*\colon H^*(B(\xi);\F_2)\longrightarrow H^*(B;\F_2)$ is an injection.
Consequently, according to Proposition \ref{prop : 1-1 in cohomology of wreath product}, the induced map between wreath squares
\[
f\wr\Z/2\colon  B\wr\Z/2\longrightarrow  B(\xi)\wr\Z/2
\]
induces a monomorphism in cohomology.
Hence equality \eqref{equality-004} implies that
\[
u(\xi)=\sum_{0\leq r< s\leq n}Q(w_r(\xi)\otimes w_s(\xi))+
\sum_{0\leq r\leq n} P(w_r(\xi))\cdot(1+t)^{n-r},
\]
concluding the proof of theorem.
\end{proof}

\section{The index of the oriented Grassmann manifold }
\label{sec : proof of the main theorem - 01}

In this section we utilize the definition of the oriented Grassmann manifold as a homogeneous space to obtain an alternative description of the index that will further on allow us to give a complete proof of Theorem \ref{th : main - 01} in Section \ref{sec : proofs }.
In particular, at the end of this section we prove Theorem \ref{th : main - 01} in the case when $n=4$ and  $n=6$.

For this section we assume that $n\geq 2$ is an even integer.
Consequently there exist unique integers $a\geq 1$ and $b\geq 0$ such that $n=2^a(2b+1)$.

\subsection{A useful description of the index of the oriented Grassmann manifold}
\label{sub : A useful description of the index of the oriented Grassmann manifold}
The real oriented Grassmann manifold  $\widetilde{G}_{n}(\R^{2n})$ can be defined by
\[
\widetilde{G}_{n}(\R^{2n}):=\SO(2n)/(\SO(n)\times\SO(n)).
\]
The $\Z/2=\langle\omega\rangle$ action on $\widetilde{G}_{n}(\R^{2n})$ we consider is defined by sending an oriented $n$-dimensional subspace $V$ to its (appropriately oriented) orthogonal complement $V^{\perp}$.

\medskip
Let $\Z/2$ acts on the product $\SO(n)\times\SO(n)$ by interchanging the factors, and let
\begin{equation}
	\label{eq : def of W}
W:= (\SO(n)\times\SO(n))\rtimes\Z/2
\end{equation}
be the semi-direct product induced by this action.
In other words $W$ is the wreath product $\SO(n)\wr\Z/2$.
Thus we have an exact sequence of groups
\begin{equation}
	\label{eq : exact seq of groups - 01}
	\xymatrix{
	1\ar[r] &  \SO(n)\times\SO(n)\ar[r] & W\ar[r]^-{p}  & \Z/2\ar[r] & 1.
	}
\end{equation}
The exact sequence of groups \eqref{eq : exact seq of groups - 01} induces a Lyndon--Hochschild--Serre spectral sequence \cite[Sec.\,IV.1]{Adem2004} whose $E_2$ term is given by
\[
E^{i,j}_2(W)=H^i(\B(\Z/2);\mathcal{H}^j(\B(\SO(n)\times\SO(n));\F_2))\cong H^i(\Z/2;H^j(\SO(n)\times\SO(n);\F_2)),
\]
and that converges to the cohomology $H^*(W;\F/2)$.
From a classical result of Nakaoka \cite{Nakaoka1961}, or from  \cite[Thm.\,1.7]{Adem2004}, we have that this spectral sequence collapses at $E_2$-term.
Consequently, the induced map in cohomology
\[
\B(p)^*\colon H^*(\B(\Z/2);\F_2)\longrightarrow H^*(W;\F_2)
\]
is injection.

\medskip
The group $W$ can be seen as a subgroup of $\SO(2n)$ via the following embedding
\[
 (A,B)  \longmapsto \left(
\begin{array}{cc}
	A & 0\\
	0 & B
\end{array}
\right)
,
\qquad
\omega\longmapsto \left(
\begin{array}{cc}
	0 & I\\
	I & 0
\end{array}
\right)
,
\]
where $(A,B)\in \SO(n)\times\SO(n)$ and $\omega\in\Z/2$.
A direct computation verifies that
\[
\left(
\begin{array}{cc}
	A & 0\\
	0 & B
\end{array}
\right)
\left(
\begin{array}{cc}
	0 & I\\
	I & 0
\end{array}
\right)
=
\left(
\begin{array}{cc}
	0 & I\\
	I & 0
\end{array}
\right)
\left(
\begin{array}{cc}
	B & 0\\
	0 & A
\end{array}
\right) .
\]
Thus, there is a homeomorphism between the quotient spaces
\[
\widetilde{G}_{n}(\R^{2n})/(\Z/2)\cong \SO(2n)/W.
\]
Since the action of any subgroup on the ambient group by multiplication is free we have homotopy equivalences between quotient spaces and associated Borel constructions
\begin{equation}
\label{eq : homotopy of quotient and Borel construction}
\E(\Z/2)\times_{\Z/2}\widetilde{G}_{n}(\R^{2n})  \simeq\widetilde{G}_{n}(\R^{2n})/(\Z/2)\cong \SO(2n)/W \simeq \E\SO(2n)\times_W\SO(2n).
\end{equation}

\medskip
The space $\E\SO(2n)$, as a free contractible  $\SO(2n)$-space, can be seen also as a model for $\E W$ via the inclusion map $i\colon W\longrightarrow \SO(2n)$.
Thus we have the morphism of Borel construction fibrations presented by the right commutative square of the diagram
\begin{equation}
	\label{diagram - 01}
	\xymatrix{
	\E(\Z/2)\times_{\Z/2}\widetilde{G}_{n}(\R^{2n})\ar[d]_{\pi}&\E\SO(2n)\times_W\SO(2n)\ar[rr]\ar[d]_{\pi_1}\ar[l] &  & \E\SO(2n)\times_{\SO(2n)}\SO(2n)\ar[d]_{\pi_2} \\
	\B(\Z/2)&\B W\ar[rr]^-{\B(i)}\ar[l]_-{\B(p)} &  & \B\SO(2n).
	}
\end{equation}
The left square of the diagram is induced by the homotopy \eqref{eq : homotopy of quotient and Borel construction} and commutes up to a homotopy.
Notice that
\[
\E\SO(2n)\times_{\SO(2n)}\SO(2n)\simeq \SO(2n)/\SO(2n) \cong \{ \text{pt}\}.
\]

From the homotopy equivalence \eqref{eq : homotopy of quotient and Borel construction} and the fact that $\B(p)^*$ is an injection we get the following alternative description of the index of the oriented Grassmann manifold
\begin{equation}
	\label{eq : alternative description}
	\ind_{\Z/2}(\widetilde{G}_{n}(\R^{2n});\F_2)=\ker(\pi^*)\cong\ker(\pi_1^*)\cap \im \B(p)^*.
\end{equation}

\subsection{Description of the kernel of $\pi_1^*$}

In order to describe the kernel we study the morphism of Serre spectral sequences \eqref{diagram - 01}.
First we recall that
\begin{equation}
\label{eq : cohomology of SO(2n)}
H^*(\SO(2n);\F_2)\cong \F_2[e_1,e_3,\ldots,e_{2n-1}]/\langle e_1^{\alpha_1},e_3^{\alpha_3},\ldots,e_{2n-1}^{\alpha_{2n-1}}\rangle	,
\end{equation}
where $\deg(e_i)=i$, and $\alpha_i$ is the smallest power of two with the property that $i\alpha_i\geq 2n$, consult for example \cite[Thm.\,3D.2]{Hatcher2002}.
Furthermore, the set $\{e_1,e_2,e_3,\ldots,e_{2n-2},e_{2n-1}\}$, where $e_{2i}:=e_i^2$ for $i<n$, forms a simple system of  generators.
Then the cohomology of the associated classifying space can be evaluated to be
\begin{equation}
\label{eq : cohomology of BSO(2n)}
H^*(\B\SO(2n);\F_2)\cong\F_2[w_2,w_3,\ldots,w_{2n}],	
\end{equation}
where $w_j=w_j(c^*\gamma^{2n}(\R^{\infty}))$ for all $2\leq j\leq 2n$.
For more details see for example \cite[p.\,216]{McCleary2001} or \cite[Thm.\,III.3.19]{Mimura1991}.

\medskip
The Serre spectral sequence of the fibration
\begin{equation}
\label{eq : fibrtion - 001}
\xymatrix{
\SO(2n)\ar[r] & \E\SO(2n)\times_{\SO(2n)}\SO(2n)\ar[r]  &  \B\SO(2n)
}	
\end{equation}
has the $E_2$-term given by
\[
E^{i,j}_2(\E\SO(2n)\times_{\SO(2n)}\SO(2n)) = H^i(\B\SO(2n);\mathcal{H}^j(\SO(2n);\F_2)),
\]
where the local coefficient system is determined by the action of $\pi_1(\B\SO(2n))$.
Since $\pi_1(\B\SO(2n))=0$ the cohomology with local coefficients we consider becomes ordinary cohomology, and then the Universal Coefficients theorem gives the isomorphism
\[
E^{i,j}_2(\E\SO(2n)\times_{\SO(2n)}\SO(2n)) \cong H^i(\B\SO(2n);\F_2)\otimes H^j(\SO(2n);\F_2).
\]
The spectral sequence converges to $H^*(\E\SO(2n)\times_{\SO(2n)}\SO(2n);\F_2)\cong H^*(\pt;\F_2)$, and therefore
\[
E^{i,j}_{\infty}=\begin{cases}
	\F_2, & \text{for }i=0\text{ and }j=0,\\
	0, & \text{otherwise}.
\end{cases}
\]
Now applying the classical work of Borel \cite{Borel1953Annals} and Zeeman \cite[Thm.\,2]{Zeeaman1958} we have that all generators $e_1,e_2,e_3,\ldots,e_{2n-2},e_{2n-1}$  are transgressive with $\partial_k(e_{k-1})=w_k$ for all $2\leq k\leq 2n$.

\medskip
The Serre spectral sequence of the fibration
\begin{equation}
\label{eq : fibrtion - 002}
\xymatrix{
\SO(2n)\ar[r] & \E\SO(2n)\times_{W}\SO(2n)\ar[r]  &  \B W
}	
\end{equation}
has the $E_2$-term given by
\[
E^{i,j}_2(\E\SO(2n)\times_{W}\SO(2n)) = H^i(\B W;\mathcal{H}^j(\SO(2n);\F_2)),
\]
where the local coefficient system is determined by the action of $\pi_1(\B W)$.
Since the action of $\pi_1(\B W)$ is determined via the inclusion $\pi_1(\B W)\longrightarrow \pi_1(\B\SO(2n))$ and $\pi_1(\B\SO(2n))=0$ we have that the cohomology with local coefficients we consider is just the ordinary cohomology due to the Universal Coefficients theorem.
Thus
\[
E^{i,j}_2(\E\SO(2n)\times_{W}\SO(2n)) = H^i(\B W;\mathcal{H}^j(\SO(2n);\F_2))\cong H^i(\B W;\F_2)\otimes H^j(\SO(2n);\F_2).
\]

\medskip
The morphism \eqref{diagram - 01} between fibrations \eqref{eq : fibrtion - 002} and \eqref{eq : fibrtion - 001}, that have identical fibers, induces a morphism between the corresponding Serre spectral sequences
\[
E^{*,*}_*(\E\SO(2n)\times_{W}\SO(2n)) \longleftarrow E^{*,*}_*(\E\SO(2n)\times_{\SO(2n)}\SO(2n))
\]
that is the identity on the  zero column of the $E_2$-terms.
Since differentials of spectral sequences commute with morphisms of spectral sequences we have that in the spectral sequence $E^{*,*}_*(\E\SO(2n)\times_{W}\SO(2n))$ differential have the following values on the system of simple generators $\partial_k(e_{k-1})=\B(i)^*(w_k)$ for all $2\leq k\leq 2n$.
In particular, we have described $\ker(\pi_1^*)$ as the following ideal
\begin{equation}
	\label{eq : kernel}
	\ker(\pi_1^*)=\langle \B(i)^*(w_2),\B(i)^*(w_3)\ldots,\B(i)^*(w_{2n})\rangle\subseteq H^*(\B W;\F_2) .
\end{equation}

\medskip
In order to further describe $\ker(\pi_1^*)$ as an ideal we identify its generators $\B(i)^*(w_k)$ as Stiefel--Whitney classes of an appropriate pull-back vector bundle.
Since $w_k=w_k(c^*\gamma^{2n}(\R^{\infty}))$ we have that
\begin{equation}
	\label{eq : generatros as SW classes}
\B(i)^*(w_k)=\B(i)^*(w_k(c^*\gamma^{2n}(\R^{\infty})))=w_k ((c\circ\B(i))^*(\gamma^{2n}(\R^{\infty}))
\end{equation}
The vector bundle $(c\circ\B(i))^*(\gamma^{2n}(\R^{\infty}))$ can be described as follows.
The  canonical vector bundle $\gamma^{2n}(\R^{\infty})$ can be presented as the following Borel construction
\[
\xymatrix{
\R^{2n}\ar[r]   &  \E\OO(2n)\times_{\OO(2n)}\R^{2n}\ar[r]   &   \B\OO(2n).
}
\]
Consequently we have pull-back bundles $c^*\gamma^{2n}(\R^{\infty})$ and $(c\circ\B(i))^*(\gamma^{2n}(\R^{\infty}))$ also presented as Borel constructions
\[
\xymatrix{
\E\OO(2n)\times_{W}\R^{2n}\ar[r]\ar[d]   & \E\OO(2n)\times_{\SO(2n)}\R^{2n}\ar[r]\ar[d]   & \E\OO(2n)\times_{\OO(2n)}\R^{2n}\ar[d]\\
\B W\ar[r]  &   \B\SO(2n)\ar[r]    &   \B\OO(2n).
}
\]
Thus in order to completely describe $\ker(\pi_1^*)$ we need to compute Stiefel--Whitney classes of the vector bundle $(c\circ\B(i))^*(\gamma^{2n}(\R^{\infty}))$:
\begin{equation}
	\label{vector bundle 01}
	\xymatrix{
\R^{2n}\ar[r]   &  \E\OO(2n)\times_{W}\R^{2n}\ar[r]   &   \B W.
}
\end{equation}

\subsection{Cohomology of the group $W$}
\label{sec : Cohomology of the group $W$}
The group $W$ was defined in \eqref{eq : def of W} to be the wreath product $\SO(n)\wr\Z/2=(\SO(n)\times\SO(n))\rtimes\Z/2$.
As we have seen the exact sequence of the groups \eqref{eq : exact seq of groups - 01},
\[
	\xymatrix{
	1\ar[r] &  \SO(n)\times\SO(n)\ar[r] & W\ar[r]^-{p}  & \Z/2\ar[r] & 1,
	}
\]
induced a Lyndon--Hochschild--Serre spectral sequence \cite[Sec.\,IV.1]{Adem2004} whose $E_2$-term is given by
\begin{equation}
	\label{eq : spectral sequence - 002}
	E^{i,j}_2(W)=H^i(\B(\Z/2);\mathcal{H}^j(\B(\SO(n)\times\SO(n));\F_2))\cong H^i(\Z/2;H^j(\SO(n)\times\SO(n);\F_2)).
\end{equation}
This spectral sequence converges to $H^*(\B W;\Z/2)$.
A classical work of Nakaoka \cite{Nakaoka1961}, or for example \cite[Thm.\,IV.1.7]{Adem2004}, implies that this spectral sequence collapses at the $E_2$-term, that means $E^{*,*}_2(W)\cong E^{*,*}_{\infty}(W)$.
We now describe the $E_2$-term, and therefore $E_{\infty}$-term, in more details.
For that we use \cite[Lem.\,IV.1.4]{Adem2004} and its consequence \cite[Cor.\,IV.1.6]{Adem2004}, see also discussion in Sections \ref{sub : Wreath squares} and \ref{sub : Cohomology of the wreath square of a space}.
Note that we here use the fact that the classifying space $\B W$ of the group $W$ can be modeled by the wreath square of $\B\SO(n)$.
For an illustration of this spectral sequence see Figure \ref{fig 2: E_2 term}.

\begin{figure}
\centering
\includegraphics[scale=1]{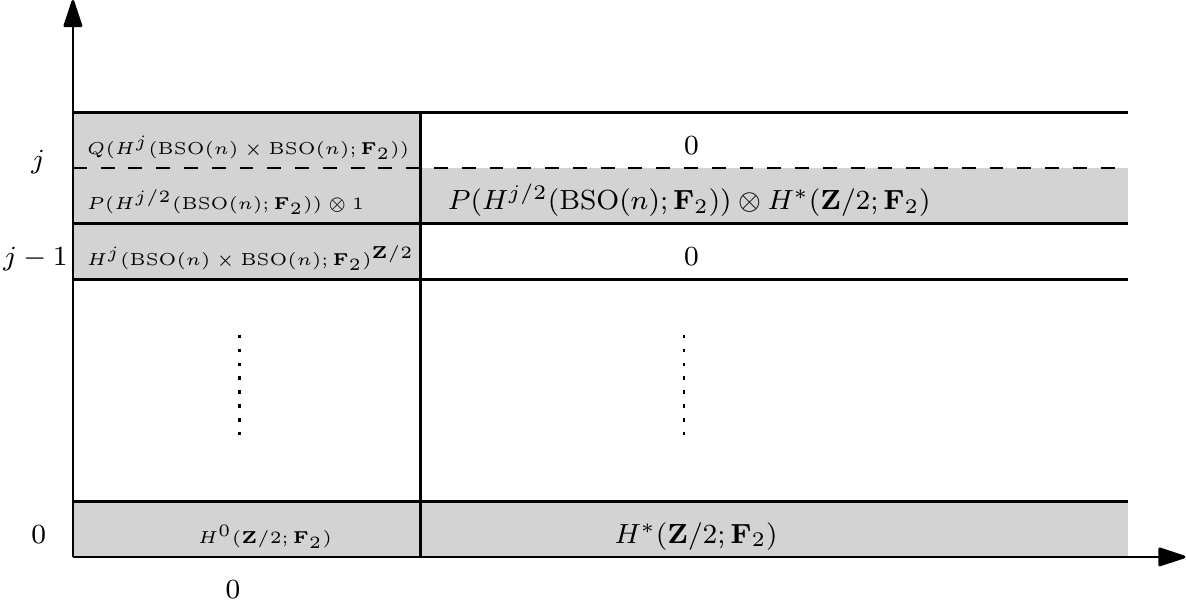}
\caption{\small The $E_2$-term of the spectral sequence \eqref{eq : spectral sequence - 002}. }
\label{fig 2: E_2 term}
\end{figure}

\medskip
Now we apply \cite[Cor.\,IV.1.6]{Adem2004} to our spectral sequence~\eqref{eq : spectral sequence - 002}.
Let us denote
\[
E^{i,0}_2(W)\cong E^{i,0}_{\infty}(W)\cong H^i(\Z/2;\F_2)\cong\F_2[t],
\]
where as before $\deg(t)=1$.
Then we have that
\[
	E^{0,j}_2(W)\cong E^{0,j}_{\infty}(W) \cong  H^{j}(\B\SO(n)\times \B\SO(n);\F_2)^{\Z/2}.
\]
In the case when $j\geq 2$ is even description can be made more precise
\[
E^{0,j}_2(W)\cong E^{0,j}_{\infty}(W) \cong P(H^{j/2}(\B\SO(n);\F_2))\oplus Q(H^{j}(\B\SO(n)\times \B\SO(n);\F_2)).
\]
Furthermore, still for $j\geq 2$ even and $i\geq 1$ we have that
\[
E^{i,j}_2(W)\cong E^{i,j}_{\infty}(W) \cong P(H^{j/2}(\B\SO(n);\F_2))\otimes H^i(\Z/2;\F_2).
\]
From the combination with the previous description of the spectral sequence, multiplicative property of the map $P$ and appropriate formula for $Q$, and the relation
\begin{equation}
	\label{eq : multiplication Q and t}
	Q(H^{j}(\B\SO(n)\times \B\SO(n);\F_2))\cdot  t=0,
\end{equation}
we get a description of the multiplicative structure of the $E_{\infty}$-term as well as $H^*(\B W;\Z/2)$.
For the fact that no multiplicative extension problem arrises see for example \cite[Rem. after Thm.\,2.1]{Leary1997}.

\subsection{Stiefel--Whitney classes of \eqref{vector bundle 01} and proof of particular cases of Theorem \ref{th : main - 01}}
In the process of describing $\ker(\pi_1^*)$ we obtained in \eqref{eq : kernel} that this kernel is an ideal in $H^*(W;\F_2)$ generated by the elements $\B(i)^*(w_2),\ldots,\B(i)^*(w_{2n})$.
Furthermore, in \eqref{eq : generatros as SW classes} we interpreted the generators of the kernel ideal to be the Stiefel--Whitney classes of the vector bundle $(c\circ\B(i))^*(\gamma^{2n}(\R^{\infty}))$:
\[
\xymatrix{
\R^{2n}\ar[r]   &  \E\OO(2n)\times_{W}\R^{2n}\ar[r]   &   \B W.
}
\]
It is critical to observe that the vector bundle $(c\circ\B(i))^*(\gamma^{2n}(\R^{\infty}))$ is isomorphic to the wreath square vector bundle $c^*\gamma^n(\R^{\infty})\wr\Z/2$.
Consequently we get a formula for the computation of the Stiefel--Whitney classes of the vector bundle \eqref{vector bundle 01} in $H^*(W;\F_2)$ as a direct consequence of the general formula derived in Theorem \ref{thm : SW classes of the wreath product}.

\begin{proposition}
	With the notation and assumptions already made the total Stiefel--Whitney class of the vector bundle $(c\circ\B(i))^*(\gamma^{2n}(\R^{\infty}))$ is
\begin{multline}
\label{XXX}
w((c\circ\B(i))^*(\gamma^{2n}(\R^{\infty})))	=\hfill\\
\sum_{0\leq r< s\leq n}Q(w_r(c^*\gamma^{n}(\R^{\infty}))\otimes w_s(c^*\gamma^{n}(\R^{\infty})))+
\sum_{0\leq r\leq n} P(w_r(c^*\gamma^{n}(\R^{\infty})))\cdot(1+t)^{n-r}.
\end{multline}
\end{proposition}

\begin{example}
For $n=2$ using the relation \eqref{XXX} we can give the total Stiefel--Whitney class of the vector bundle $(c\circ\B(i))^*(\gamma^{4}(\R^{\infty}))$ in the following form
\begin{multline*}
w((c\circ\B(i))^*(\gamma^{4}(\R^{\infty})))	=
\sum_{0\leq r< s\leq 2}Q(w_r\otimes w_s)+
\sum_{0\leq r\leq 2} P(w_r)\cdot(1+t)^{2-i}\hfill\\
= Q(w_0\otimes w_1)+Q(w_0\otimes w_2)+Q(w_1\otimes w_2)+P(w_0)(1+t)^2+P(w_1)(1+t)+P(w_2).
\end{multline*}
Here we use simplified notation $w_i:=w_i(c^*\gamma^{2}(\R^{\infty}))$.
Since $c^*$ pulls-back the tautological bundle $\gamma^{2}(\R^{\infty})$ to the oriented Grassmann manifold we have that $w_1=0$.
Therefore, we have
\begin{multline*}
w((c\circ\B(i))^*(\gamma^{4}(\R^{\infty})))	= Q(w_0\otimes w_2) + P(w_0)(1+t)^2 +P(w_2)\hfill \\
= Q(w_0\otimes w_2) + P(w_0)+P(w_0)t^2 +P(w_2)
=1+ (Q(w_0\otimes w_2)+t^2)+P(w_2).
\end{multline*}
Consequently, from \eqref{eq : kernel} we get that
\[
\ker(\pi_1^*)=\langle Q(w_0\otimes w_2)+t^2, P(w_2)\rangle.
\]
In particular, the multiplication property \eqref{eq : multiplication Q and t} implies that $t\cdot (Q(w_0\otimes w_2)+t^2)=t^3\in \ker(\pi_1^*)$.
Now, alternative description of the index \eqref{eq : alternative description} implies that
\[
\ind_{\Z/2}(\widetilde{G}_{2}(\R^{4});\F_2)=\ker(\pi_1^*)\cap \im \B(p)^*=\langle t^3\rangle.
\]
Thus Theorem \ref{th : main - 01} holds for $n=2$.
\end{example}

\begin{example}
	For $n=4$ let $w_i:=w_i(c^*\gamma^{4}(\R^{\infty}))$.
	As in the previous example $w_1=0$ because $c^*$ pulls-back the tautological bundle $\gamma^{4}(\R^{\infty})$ to the oriented Grassmann manifold.
	The relation \eqref{XXX} in this situation reads as follows
\begin{multline*}
w((c\circ\B(i))^*(\gamma^{8}(\R^{\infty})))	=
\sum_{0\leq r< s\leq 4}Q(w_r\otimes w_s)+
\sum_{0\leq r\leq 4} P(w_r)\cdot(1+t)^{4-r}=\hfill\\
\sum_{2\leq s\leq 4}Q(w_0\otimes w_s)+\sum_{2\leq r< s\leq 4}Q(w_r\otimes w_s)+
(1+t)^4+P(w_2)(1+t)^2+P(w_3)(1+t)+P(w_4)=\\
1+ Q(w_0\otimes w_2) + Q(w_0\otimes w_3) + (t^4+P(w_2)+ Q(w_0\otimes w_4))+ Q(w_2\otimes w_3) + \\(P(w_2)t^2+P(w_3)+Q(w_2\otimes w_4))+(P(w_3)t+Q(w_3\otimes w_4))+P(w_4).
\end{multline*}
Thus,
\begin{multline*}
\ker(\pi_1^*)=
\langle
Q(w_0\otimes w_2), \, Q(w_0\otimes w_3),\, t^4+P(w_2)+ Q(w_0\otimes w_4),\, Q(w_2\otimes w_3), \\ P(w_2)t^2+P(w_3)+Q(w_2\otimes w_4),\, P(w_3)t+Q(w_3\otimes w_4),\, P(w_4)
\rangle.
\end{multline*}
We do the following computation in the ideal $J:=\ker(\pi_1^*)$ with the multiplication property \eqref{eq : multiplication Q and t} in mind:
\begin{eqnarray*}
	 t^4+P(w_2)+ Q(w_0\otimes w_4)\in J &\Longrightarrow &  t^6+P(w_2)t^2 \in J,\\
	 t^6+P(w_2)t^2 \in J\quad\text{and}\quad P(w_2)t^2+P(w_3)+Q(w_2\otimes w_4)\in J &\Longrightarrow & t^6+P(w_3)+Q(w_2\otimes w_4)\in J,\\
	 t^6+P(w_3)+Q(w_2\otimes w_4)\in J &\Longrightarrow & t^7+P(w_3)t\in J,\\
	 t^7+P(w_3)t\in J\quad\text{and}\quad P(w_3)t+Q(w_3\otimes w_4)\in J &\Longrightarrow & t^7+Q(w_3\otimes w_4)\in J,\\
	 t^7+Q(w_3\otimes w_4)\in J  &\Longrightarrow & t^8\in J.
\end{eqnarray*}
Hence $t^7\notin J$ while $t^8\in J$.
Therefore
\[
\ind_{\Z/2}(\widetilde{G}_{4}(\R^{8});\F_2)=\ker(\pi_1^*)\cap \im \B(p)^*=\langle t^8\rangle.
\]
This concludes a proof of Theorem \ref{th : main - 01} for $n=4$.
\end{example}

\begin{example}
Let $n=6$, and let $w_i:=w_i(c^*\gamma^{6}(\R^{\infty}))$.
Again, since $c^*$ pulls-back the tautological bundle $\gamma^{6}(\R^{\infty})$ to the oriented Grassmann manifold we have that $w_1=0$.
Now from the relation \eqref{XXX} we get that
\begin{multline*}
w((c\circ\B(i))^*(\gamma^{12}(\R^{\infty})))	=
\sum_{0\leq r< s\leq 6}Q(w_r\otimes w_s)+
\sum_{0\leq r\leq 6} P(w_r)\cdot(1+t)^{6-r}=\hfill\\
\sum_{2\leq s\leq 6}Q(w_0\otimes w_s)+\sum_{2\leq r< s\leq 6}Q(w_r\otimes w_s)+\hfill \\
(1+t)^6+P(w_2)(1+t)^4+P(w_3)(1+t)^3+P(w_4)(1+t)^2+P(w_5)(1+t)+P(w_6)=\\
\sum_{2\leq s\leq 6}Q(w_0\otimes w_s)+\sum_{2\leq r< s\leq 6}Q(w_r\otimes w_s)+1+t^2+t^4+t^6\hfill\\
P(w_2)+P(w_2)t^4+P(w_3)+P(w_3)t+P(w_3)t^2+P(w_3)t^3+P(w_4)+P(w_4)t^2+P(w_5)+P(w_5)t+P(w_6).
\end{multline*}

In particular, we get that
\[
	w_2((c\circ\B(i))^*(\gamma^{12}(\R^{\infty}))) =
	Q(w_0\otimes w_2)+ t^2.
\]
The multiplication property \eqref{eq : multiplication Q and t} and 	the alternative description of the index \eqref{eq : alternative description} imply $t^2\notin \ker(\pi_1^*)$,
and furthermore that
\[
t^3=t(Q(w_0\otimes w_2)+ t^2)=t w_2((c\circ\B(i))^*(\gamma^{12}(\R^{\infty})))\in \ker(\pi_1^*)\cap \im \B(p)^*.
\]
Consequently,
\[
\ind_{\Z/2}(\widetilde{G}_{6}(\R^{12});\F_2)=\ker(\pi_1^*)\cap \im \B(p)^*=\langle t^3\rangle,
\]
and Theorem \ref{th : main - 01} holds when $n=6$.
\end{example}

\section{The index of the Grassmann manifold }
\label{sec : proof of the main theorem - 02}

Similarly as in Section \ref{sec : proof of the main theorem - 01} for the computation of the index we utilize the definition of the Grassmann manifold $G_{n}(\R^{2n})$ as a homogeneous space.
This new description of the index will be used in Section \ref{sec : proofs } for the proof of Theorem \ref{th : main - 02}.

In this section we assume that $n\geq 1$ is an arbitrary positive integer, and $a\geq 0$ and $b\geq 0$ are unique integers such that $n=2^a(2b+1)$.

\subsection{A useful description of the index of the Grassmann manifold}
The real Grassmann manifold  $G_{n}(\R^{2n})$ can be defined by
\[
G_{n}(\R^{2n}):=\OO(2n)/(\OO(n)\times\OO(n)).
\]
Recall, the $\Z/2=\langle\omega\rangle$ action on $G_{n}(\R^{2n})$ we study is given by sending an $n$-dimensional subspace $V$ to its orthogonal complement $V^{\perp}$.

\medskip
Let $\Z/2$ acts on the product $\OO(n)\times\OO(n)$ by permuting the factors, and let
\begin{equation}
	\label{eq : def of W1}
U:= (\OO(n)\times\OO(n))\rtimes\Z/2
\end{equation}
be the semi-direct product induced by this action.
Thus, $U$ is the wreath product $\OO(n)\wr\Z/2$ and there is an exact sequence of groups
\begin{equation}
	\label{eq : exact seq of groups - 011}
	\xymatrix{
	1\ar[r] &  \OO(n)\times\OO(n)\ar[r] & U\ar[r]^-{p}  & \Z/2\ar[r] & 1.
	}
\end{equation}
This exact sequence induces a Lyndon--Hochschild--Serre spectral sequence \cite[Sec.\,IV.1]{Adem2004} whose $E_2$-term is given by
\[
E^{i,j}_2(U)=H^i(\B(\Z/2);\mathcal{H}^j(\B(\OO(n)\times\OO(n));\F_2))\cong H^i(\Z/2;H^j(\OO(n)\times\OO(n);\F_2)),
\]
and that converges to the cohomology $H^*(\B U;\F_2)$.
As we have already seen the classical result of Nakaoka \cite{Nakaoka1961} (see also \cite[Thm.\,IV.1.7]{Adem2004}) implies that this spectral sequence collapses at $E_2$-term.
Hence, the induced map in cohomology
\[
\B(p)^*\colon H^*(\B(\Z/2);\F_2)\longrightarrow H^*(\B U;\F_2)
\]
has to be injective.

\medskip
The group $U$ can be seen as a subgroup of $\OO(2n)$ via the following embedding
\[
 (A,B)  \longmapsto \left(
\begin{array}{cc}
	A & 0\\
	0 & B
\end{array}
\right)
,
\qquad
\omega\longmapsto \left(
\begin{array}{cc}
	0 & I\\
	I & 0
\end{array}
\right)
,
\]
where $(A,B)\in \OO(n)\times\OO(n)$ and $\omega\in\Z/2$.
A direct computation verifies that
\[
\left(
\begin{array}{cc}
	A & 0\\
	0 & B
\end{array}
\right)
\left(
\begin{array}{cc}
	0 & I\\
	I & 0
\end{array}
\right)
=
\left(
\begin{array}{cc}
	0 & I\\
	I & 0
\end{array}
\right)
\left(
\begin{array}{cc}
	B & 0\\
	0 & A
\end{array}
\right) .
\]
Consequently, there is a homeomorphism between the quotient spaces
\[
G_{n}(\R^{2n})/(\Z/2)\cong \OO(2n)/U.
\]
Since the action of any subgroup on the ambient group by multiplication is free we have homotopy equivalences between quotient spaces and associated Borel constructions
\begin{equation}
\label{eq : homotopy of quotient and Borel construction1}
\E(\Z/2)\times_{\Z/2}G_{n}(\R^{2n})  \simeq G_{n}(\R^{2n})/(\Z/2)\cong \OO(2n)/U \simeq \E\OO(2n)\times_U\OO(2n).
\end{equation}

\medskip
The space $\E\OO(2n)$ is a free contractible $\OO(2n)$-space.
Thus it can be treated as a model for $\E U$ via the inclusion map $i\colon U\longrightarrow \OO(2n)$.
We have the morphism of Borel construction fibrations presented in the right commutative square of the diagram
\begin{equation}
	\label{diagram - 011}
	\xymatrix{
	\E(\Z/2)\times_{\Z/2}G_{n}(\R^{2n}) \ar[d]_{\pi}&\E\OO(2n)\times_U\OO(2n)\ar[rr]\ar[d]_{\pi_1}\ar[l] &  & \E\OO(2n)\times_{\OO(2n)}\OO(2n)\ar[d]_{\pi_2} \\
	\B(\Z/2)&\B U\ar[rr]^-{\B(i)}\ar[l]_-{\B(p)} &  & \B\OO(2n).
	}
\end{equation}
The left square of the diagram is induced by the homotopy \eqref{eq : homotopy of quotient and Borel construction1} and commutes up to a homotopy.
Observe that
\[
\E\OO(2n)\times_{\OO(2n)}\OO(2n)\simeq \OO(2n)/\OO(2n) \cong \{ \text{pt}\}.
\]

The homotopy equivalence \eqref{eq : homotopy of quotient and Borel construction1} and the fact that $\B(p)^*$ is an injection yield the following alternative description of the index of the Grassmann manifold
\begin{equation}
	\label{eq : alternative description1}
	\ind_{\Z/2}(G_{n}(\R^{2n});\F_2)=\ker(\pi^*)\cong\ker(\pi_1^*)\cap \im \B(p)^*.
\end{equation}

\subsection{Description of the kernel of $\pi_1^*$}

For the description of the kernel we analyze  the morphism of Serre spectral sequences \eqref{diagram - 011}.
We first recall that
\begin{equation}
\label{eq : cohomology of SO(2n)1}
H^*(\OO(2n);\F_2)\cong H^*(\SO(2n);\F_2)\oplus H^*(\SO^{-}(2n);\F_2),
\end{equation}
where $\SO^{-}(2n)=\{A\in\OO(2n) : \det(A)=-1\}\cong \SO(2n)$, see \cite[Cor.\,III.3.15]{Mimura1991}.
Then the cohomology of the associated classifying space can be evaluated to be
\begin{equation}
\label{eq : cohomology of BSO(2n)}
H^*(\B\OO(2n);\F_2)\cong\F_2[w_1,w_2,w_3,\ldots,w_{2n}],	
\end{equation}
where $w_j=w_j(\gamma^{2n}(\R^{\infty}))$ for all $1\leq j\leq 2n$.
For more details see for example \cite[p.\,216]{McCleary2001} or \cite[Thm.\,III.3.19]{Mimura1991}.

\medskip
The Serre spectral sequence of the fibration
\begin{equation}
\label{eq : fibrtion - 0011}
\xymatrix{
\OO(2n)\ar[r] & \E\OO(2n)\times_{\OO(2n)}\OO(2n)\ar[r]  &  \B\OO(2n),
}	
\end{equation}
whose fiber has two path-connected components, has the $E_2$-term given by
\[
E^{i,j}_2(\E\OO(2n)\times_{\OO(2n)}\OO(2n)) = H^i(\B\OO(2n);\mathcal{H}^j(\OO(2n);\F_2)),
\]
where the local coefficient system is determined by the action of $\pi_1(\B\OO(2n))\cong \pi_0(\OO(2n))\cong\Z/2$.
The first isomorphism follows from the long exact sequence in homotopy associated to the fibration $\OO(2n)\longrightarrow \E\OO(2n)\longrightarrow  \B\OO(2n)$. 
Since $H^*(\OO(2n);\F_2)\cong H^*(\SO(2n);\F_2)\oplus H^*(\SO^-(2n);\F_2)$ and $\pi_1(\B\OO(2n))$ acts by interchanging summands we have that as a $\pi_1(\B\OO(2n))$-module
\[
H^*(\OO(2n);\F_2)\cong \mathrm{Coind}^{\OO(2n)}_{\SO(2n)}\big(H^*(\SO(2n);\F_2)\big).
\]
Thus according to the Shapiro lemma \cite[Prop.\,6.2]{Brown1994} the $E_2$-term simplifies to
\begin{multline*}
E^{i,j}_2(\E\OO(2n)\times_{\OO(2n)}\OO(2n)) = H^i(\B\OO(2n);\mathcal{H}^j(\OO(2n);\F_2))\cong\\
H^i\big(\B\OO(2n);\mathrm{Coind}^{\OO(2n)}_{\SO(2n)}\big(H^j(\SO(2n);\F_2)\big)\big)\cong 	
H^i(\B\SO(2n);H^j(\SO(2n);\F_2))\cong \\
H^i(\B\SO(2n);\F_2)\otimes H^j(\SO(2n);\F_2).
\end{multline*}
Since the spectral sequence converges to $H^*(\E\OO(2n)\times_{\OO(2n)}\OO(2n);\F_2)\cong H^*(\pt;\F_2)$ we know that
\[
E^{i,j}_{\infty}=\begin{cases}
	\F_2, & \text{for }i=0\text{ and }j=0,\\
	0, & \text{otherwise}.
\end{cases}
\]
Now applying the classical work of Borel \cite{Borel1953Annals} and Zeeman \cite[Thm.\,2]{Zeeaman1958} we have that all elements of the simple system of generators $e_1,e_2,e_3,\ldots,e_{2n-2},e_{2n-1}$ of the cohomology of $\SO(2n)$ are transgressive with $\partial_k(e_{k-1})=w_k$ for all $2\leq k\leq 2n$.

\medskip
The Serre spectral sequence of the fibration
\begin{equation}
\label{eq : fibrtion - 0021}
\xymatrix{
\OO(2n)\ar[r] & \E\OO(2n)\times_{U}\OO(2n)\ar[r]  &  \B U,
}	
\end{equation}
with disconnected fiber, has the $E_2$-term given by
\[
E^{i,j}_2(\E\OO(2n)\times_{U}\OO(2n)) = H^i(\B U;\mathcal{H}^j(\OO(2n);\F_2)),
\]
where the local coefficient system is determined by the action of $\pi_1(\B U)\cong \pi_0(U)\cong (\Z/2\times\Z/2)\rtimes\Z/2\cong D_8$.
The first isomorphism again is obtained from the long exact sequence in homotopy associated to the fibration $U\longrightarrow \E U\longrightarrow \B U$. 
Here $D_8$ denotes the Dihedral group of isometries of a square.
Let $U^+:=U\cap\SO(2n)$ be the index $2$ subgroup of orientation preserving elements of $U$.
Then as a  $\pi_1(\B U)$-module
\[
H^*(\OO(2n);\F_2)\cong \mathrm{Coind}^{U}_{U^+}\big(H^*(\SO(2n);\F_2)\big).
\]
Therefore the Shapiro lemma \cite[Prop.\,6.2]{Brown1994} allows us to simplify the $E_2$-term again
\begin{multline*}
	E^{i,j}_2(\E\OO(2n)\times_{U}\OO(2n)) = H^i(\B U;\mathcal{H}^j(\OO(2n);\F_2))\cong
	H^i\big(\B U; \mathrm{Coind}^{U}_{U^+}\big(H^j(\SO(2n);\F_2)\big)\big)\cong \\
	H^i(\B U^+;H^j(\SO(2n);\F_2))\cong H^i(\B U^+;\F_2)\otimes H^j(\SO(2n);\F_2).
\end{multline*}

\medskip
The morphism \eqref{diagram - 011} between fibrations \eqref{eq : fibrtion - 0021} and \eqref{eq : fibrtion - 0011}, that have identical fibers, induces a morphism between the corresponding Serre spectral sequences
\[
E^{*,*}_*(\E\OO(2n)\times_{U}\OO(2n)) \longleftarrow E^{*,*}_*(\E\OO(2n)\times_{\OO(2n)}\OO(2n))
\]
which is the identity on the zero column of the $E_2$-terms, but now is {\bf not} the identity on the zero row of the $E_2$-terms.
In fact, the morphism of the spectral sequence on the zero row of the $E_2$-terms is the restriction homomorphism $\res^{SO(2n)}_{U^+}\colon H^*(\B\SO(2n);\F_2)\longrightarrow H^*(\B U^+;\F_2)$ induced by the inclusion $i|_{U^+}\colon U^+\longrightarrow SO(2n)$.
Differentials of spectral sequences commute with morphisms of spectral sequences.
Hence in the spectral sequence $E^{*,*}_*(\E\OO(2n)\times_{U}\OO(2n))$ the differential have the following values on the system of simple generators $\partial_k(e_{k-1})=\B(i|_{U^+})^*(w_k)$ for all $2\leq k\leq 2n$.

\medskip
Unlike in Section \ref{sub : A useful description of the index of the oriented Grassmann manifold} we did not yet reached a description of $\ker(\pi_1^*)$.
For that we expand the diagram \eqref{diagram - 011} as follows
\begin{equation}
	\label{diagram - new}
	\xymatrix{
	\E\OO(2n)\times_U\OO(2n)\ar[rr]\ar[dd]_{\pi_1}\ar[dr] &  & \E\OO(2n)\times_{\OO(2n)}\OO(2n)\ar[dd]|!{[d]}\hole\ar[dr] &\\
	 & \E\OO(2n)\times_U\pt\ar[rr]\ar[dd]      & & \E\OO(2n)\times_{\OO(2n)}\pt\ar[dd]  \\
	\B U\ar[rr]|!{[r]}\hole\ar[dr]_-{\id} &  & \B\OO(2n)\ar[dr]_-{\id} & \\
	 &\B U\ar[rr] &  & \B\OO(2n)
	}
\end{equation}
with diagonal maps induced by the $\OO(2n)$, and also $U$, equivariant projection $\tau\colon\OO(2n)\longrightarrow\pt$.
The cube diagram of bundle morphisms \eqref{diagram - new} induces the following diagram of the corresponding spectral sequence morphisms:
\begin{equation}
\label{digram morphism of sss}
	\xymatrix{
	E^{*,*}_*(\E\OO(2n)\times_{U}\OO(2n)) & E^{*,*}_*(\E\OO(2n)\times_{\OO(2n)}\OO(2n))\ar[l] \\
	E^{*,*}_*(\E\OO(2n)\times_{U}\pt)\ar[u]^{\eta_1} & E^{*,*}_*(\E\OO(2n)\times_{\OO(2n)}\pt),\ar[l]\ar[u]^{\eta_2}
	}
\end{equation}
where the vertical maps $\eta_1$ and $\eta_2$ are induced by the $\OO(2n)$, and also $U$, equivariant projection $\tau$.
The spectral sequences in the bottom row of the diagram collapse at the $E_2$-term.
Thus, from the infinity term of the diagram \eqref{digram morphism of sss} we get the equality
\begin{equation}
	\label{eq kernel -001}
	\ker (\pi_1) = \ker (\eta_1).
\end{equation}
Now in order to understand $\ker (\eta_1)$ we specialize diagram \eqref{digram morphism of sss} to the $E_2$-term and get
\begin{equation*}
	\xymatrix{
	H^i(\B U;\mathcal{H}^j(\OO(2n);\F_2)) & H^i(\B\OO(2n);\mathcal{H}^j(\OO(2n);\F_2))\ar[l] \\
	H^i(\B U;\mathcal{H}^j(\pt;\F_2))\ar[u]^{\eta_1} & H^i(\B\OO(2n);\mathcal{H}^j(\pt;\F_2)).\ar[l]\ar[u]^{\eta_2}
	}
\end{equation*}
Furthermore, on the zero row of the $E_2$-term we have
\begin{equation}
\label{diagram ss morphism 0 row}
	\xymatrix{
	H^i(\B U;\mathrm{Coind}^{U}_{U^+}\F_2) & H^i(\B\OO(2n);\mathrm{Coind}^{\OO(2n)}_{\SO(2n)}\F_2)\ar[l] \\
	H^i(\B U;\F_2)\ar[u]^{\eta_1} & H^i(\B\OO(2n);\F_2).\ar[l]\ar[u]^{\eta_2}
	}
\end{equation}
The vertical maps in \eqref{diagram ss morphism 0 row} are determined by the coefficient morphism $\tau^*\colon H^*(\pt;\F_2)\longrightarrow H^*(\OO(2n);\F_2)$.
Since from Shapiro lemma
\[
H^i(\B U;\mathrm{Coind}^{U}_{U^+}\F_2)\cong H^i(\B U^+;\F_2)
\qquad\text{and}\qquad
H^i(\B\OO(2n);\mathrm{Coind}^{\OO(2n)}_{\SO(2n)}\F_2)\cong H^i(\B\SO(2n);\F_2)
\]
we have that $\eta_1=\res^{U}_{U^+}$ and $\eta_2=\res^{\OO(2n)}_{\SO(2n)}$.
From computation of the spectral sequences of the fibrations \eqref{eq : fibrtion - 0011} and \eqref{eq : fibrtion - 0021} we have that the diagram \eqref{diagram ss morphism 0 row} in infinity term becomes
\begin{equation}
\label{diagram ss morphism 0 row infinity term}
	\xymatrix{
	H^i(\B U^+;\F_2)/\langle \B(i|_{U^+})^*(w_2),\ldots, \B(i|_{U^+})^*(w_{2n}) \rangle & H^i(\B\SO(2n);\F_2)/\langle w_2,\ldots,w_{2n}\rangle \ar[l] \\
	H^i(\B U;\F_2)\ar[u]^{\eta_1} & H^i(\B\OO(2n);\F_2).\ar[l]\ar[u]^{\eta_2}
	}
\end{equation}
Since $\eta_1$ is the restriction,  $\B(i|_{U^+})^*(w_2),\ldots, \B(i|_{U^+})^*(w_{2n})$ are Stiefel--Whitney classes fulfilling naturality, using the sequence of group inclusions $U^+\longrightarrow U\overset{i}{\longrightarrow}\OO(2n)$ we have that
\begin{equation}
	\label{eq : kernel 2}
	\ker(\pi_1^*)=\ker(\eta_1^*)=\langle \B(i)^*(w_1), \B(i)^*(w_2),\ldots,\B(i)^*(w_{2n})\rangle\subseteq H^*(\B U;\F_2) .
\end{equation}
Observe that commutativity of the diagram \eqref{diagram ss morphism 0 row infinity term} and the naturally of Stiefel--Whitney classes force the first Stiefel--Whitney class $\B(i)^*(w_1)$ in the kernel.

\medskip
Thus in order to completely describe $\ker(\pi_1^*)$ we need to compute Stiefel--Whitney classes of the vector bundle $ B(i)^*(\gamma^{2n}(\R^{\infty}))$:
\begin{equation}
	\label{vector bundle 011}
	\xymatrix{
\R^{2n}\ar[r]   &  \E\OO(2n)\times_{U}\R^{2n}\ar[r]   &   \B U.
}
\end{equation}

\subsection{Cohomology of the groups $U$}
\label{sec : Cohomology of the group U}

The group $U$ was defined to be the wreath product $\OO(n)\wr\Z/2=(\OO(n)\times\OO(n))\rtimes\Z/2$.
The exact sequence of the groups \eqref{eq : exact seq of groups - 011},
\[
	\xymatrix{
	1\ar[r] &  \OO(n)\times\OO(n)\ar[r] & U\ar[r]^-{p}  & \Z/2\ar[r] & 1,
	}
\]
induced a Lyndon--Hochschild--Serre spectral sequence whose $E_2$-term is given by
\begin{equation}
	\label{eq : spectral sequence - 02}
	E^{i,j}_2(U)=H^i(\B(\Z/2);\mathcal{H}^j(\B(\OO(n)\times\OO(n));\F_2))\cong H^i(\Z/2;H^j(\OO(n)\times\OO(n);\F_2)).
\end{equation}
This spectral sequence converges to $H^*(\B U;\Z/2)$.
As we have seen this spectral sequence collapses at the $E_2$-term, meaning $E^{*,*}_2(U)\cong E^{*,*}_{\infty}(U)$.
Like in Section \ref{sec : Cohomology of the group $W$} we describe the $E_2$-term, and therefore $E_{\infty}$-term, in more details.
For that we use again \cite[Lem.\,IV.1.4]{Adem2004} and \cite[Cor.\,IV.1.6]{Adem2004}, consult also Sections \ref{sub : Wreath squares} and \ref{sub : Cohomology of the wreath square of a space}.
We utilize the fact that the classifying space $\B U$ of the group $U$ can be obtained to be the wreath square of $\B\OO(n)$.

\medskip
Let us apply \cite[Cor.\,IV.1.6]{Adem2004} to our spectral sequence~\eqref{eq : spectral sequence - 02}.
Denote
\[
E^{i,0}_2(U)\cong E^{i,0}_{\infty}(U)\cong H^i(\Z/2;\F_2)\cong\F_2[t],
\]
where as before $\deg(t)=1$.
Then we have that
\[
	E^{0,j}_2(U)\cong E^{0,j}_{\infty}(U) \cong  H^{j}(\B\OO(n)\times \B\OO(n);\F_2)^{\Z/2}.
\]
In particular when $j\geq 2$ is even description can be made even more precise
\[
E^{0,j}_2(U)\cong E^{0,j}_{\infty}(U) \cong P(H^{j/2}(\B\OO(n);\F_2))\oplus Q(H^{j}(\B\OO(n)\times \B\OO(n);\F_2)).
\]
Furthermore, still for $j\geq 2$ even and $i\geq 1$ we have that
\[
E^{i,j}_2(U)\cong E^{i,j}_{\infty}(U) \cong P(H^{j/2}(\B\OO(n);\F_2))\otimes H^i(\Z/2;\F_2).
\]
From the combination with the previous description of the spectral sequence, multiplicative property of the map $P$ and appropriate formula for $Q$, and the relation
\begin{equation}
	\label{eq : multiplication Q and t 2}
	Q(H^{j}(\B\OO(n)\times \B\OO(n);\F_2))\cdot  t=0,
\end{equation}
we get a description of the multiplicative structure of the $E_{\infty}$-term as well as $H^*(\B U;\Z/2)$.
The multiplicative extension problem does not arise in this situation, for more details consult for example \cite[Rem. after Thm.\,2.1]{Leary1997}.

\subsection{Stiefel--Whitney classes of \eqref{vector bundle 011} and proof of particular cases of Theorem \ref{th : main - 02}}
In the process of describing $\ker(\pi_1^*)$ we obtained in \eqref{eq : kernel 2} that this kernel is an ideal in $H^*(\B U;\F_2)$ generated by the elements $\B(i)^*(w_1),\ldots,\B(i)^*(w_{2n})$.
Furthermore,  we interpreted the generators of the kernel ideal to be the Stiefel--Whitney classes of the vector bundle $\B(i)^*(\gamma^{2n}(\R^{\infty}))$:
\[
\xymatrix{
\R^{2n}\ar[r]   &  \E\OO(2n)\times_{U}\R^{2n}\ar[r]   &   \B U.
}
\]
As before it can be seen that the vector bundle $\B(i)^*(\gamma^{2n}(\R^{\infty}))$ is isomorphic to the wreath square vector bundle $\gamma^n(\R^{\infty})\wr\Z/2$.
Thus the following formula for the computation of the Stiefel--Whitney classes of the vector bundle \eqref{vector bundle 011} in $H^*(\B U;\F_2)$ is a direct consequence of the general formula derived in Theorem \ref{thm : SW classes of the wreath product}.

\begin{proposition}
	With the notation and assumptions already made the total Stiefel--Whitney class of the vector bundle $\B(i)^*(\gamma^{2n}(\R^{\infty}))$ is
\begin{multline}
\label{XXXY}
w(\B(i)^*(\gamma^{2n}(\R^{\infty})))	=\hfill\\
\sum_{0\leq r< s\leq n}Q(w_r(\gamma^{n}(\R^{\infty}))\otimes w_s(\gamma^{n}(\R^{\infty})))+
\sum_{0\leq r\leq n} P(w_r(\gamma^{n}(\R^{\infty})))\cdot(1+t)^{n-r}.
\end{multline}
\end{proposition}

\begin{example}
\label{ex : n=2}
For $n=2$ using the relation \eqref{XXXY} we can give the total Stiefel--Whitney class of the vector bundle $\B(i)^*(\gamma^{4}(\R^{\infty}))$ in the following form
\begin{multline*}
w(B(i)^*(\gamma^{4}(\R^{\infty})))	=
\sum_{0\leq r< s\leq 2}Q(w_r\otimes w_s)+
\sum_{0\leq r\leq 2} P(w_r)\cdot(1+t)^{2-i}\hfill\\
= Q(w_0\otimes w_1)+Q(w_0\otimes w_2)+Q(w_1\otimes w_2)+P(w_0)(1+t)^2+P(w_1)(1+t)+P(w_2).
\end{multline*}
Consequently, from \eqref{eq : kernel 2} we get that
\[
\ker(\pi_1^*)=\langle Q(w_0\otimes w_1), Q(w_0\otimes w_2)+t^2+P(w_1),Q(w_1\otimes w_2)+P(w_1)t ,P(w_2)\rangle.
\]
In particular, the multiplication property \eqref{eq : multiplication Q and t 2} implies that
\[
t\cdot (Q(w_0\otimes w_2)+t^2+P(w_1))=t^3+P(w_1)t\in \ker(\pi_1^*).
\]
Since $Q(w_1\otimes w_2)+P(w_1)t\in \ker(\pi_1^*)$ we have $t^3+Q(w_1\otimes w_2)\in \ker(\pi_1^*)$.
Again the multiplication property \eqref{eq : multiplication Q and t 2} yields
\[
t(t^3+Q(w_1\otimes w_2))=t^4\in \ker(\pi_1^*).
\]
Now, from \eqref{eq : alternative description1} follows
\[
\ind_{\Z/2}(G_{2}(\R^{4});\F_2)=\ker(\pi_1^*)\cap \im \B(p)^*=\langle t^4\rangle.
\]
Thus Theorem \ref{th : main - 02} holds for $n=2$.
\end{example}

\begin{example}
	For $n=4$ let $w_i:=w_i(\gamma^{4}(\R^{\infty}))$.
	The relation \eqref{XXXY} in this situation reads as follows
\begin{multline*}
w(\B(i)^*(\gamma^{8}(\R^{\infty})))	=
\sum_{0\leq r< s\leq 4}Q(w_r\otimes w_s)+
\sum_{0\leq r\leq 4} P(w_r)\cdot(1+t)^{4-r}=\hfill\\
\sum_{0\leq r< s\leq 4}Q(w_r\otimes w_s)+
P(w_0)(1+t)^4+P(w_1)(1+t)^3+P(w_2)(1+t)^2+P(w_3)(1+t)+P(w_4)=\hfill\\
1+ Q(w_0\otimes w_1) + \big(Q(w_0\otimes w_2) +P(w_1)\big) + \big(Q(w_0\otimes w_3)+Q(w_1\otimes w_2) +P(w_1)t\big)+   \hfill\\
\big( Q(w_0\otimes w_4)+Q(w_1\otimes w_3)+t^4+P(w_1)t^2+P(w_2) \big) + \big (Q(w_1\otimes w_4)+Q(w_2\otimes w_3) +P(w_1)t^3\big) +\\
\big(Q(w_2\otimes w_4)+P(w_2)t^2+P(w_3)\big)+\big(Q(w_3\otimes w_4)+P(w_3)t)+P(w_4).
\end{multline*}
Thus,
\begin{multline*}
\ker(\pi_1^*)=
\langle
Q(w_0\otimes w_1), \, Q(w_0\otimes w_2) +P(w_1),\, Q(w_0\otimes w_3)+Q(w_1\otimes w_2) +P(w_1)t,\\ Q(w_0\otimes w_4)+Q(w_1\otimes w_3)+t^4+P(w_1)t^2+P(w_2),\,Q(w_1\otimes w_4)+Q(w_2\otimes w_3) +P(w_1)t^3,\\
  Q(w_2\otimes w_4)+P(w_2)t^2+P(w_3),\,Q(w_3\otimes w_4)+P(w_3)t,\, P(w_4)
\rangle.
\end{multline*}
Since $ Q(w_0\otimes w_2) +P(w_1)\in J$ then $P(w_1)t \in J$.
Thus we can modify the presentation of the kernel ideal as follows
\begin{multline*}
\ker(\pi_1^*)=
\langle
Q(w_0\otimes w_1), \, Q(w_0\otimes w_2) +P(w_1),\, P(w_1)t,\, Q(w_0\otimes w_3)+Q(w_1\otimes w_2),\\ Q(w_0\otimes w_4)+Q(w_1\otimes w_3)+t^4+P(w_2),\,Q(w_1\otimes w_4)+Q(w_2\otimes w_3) ,\\
  Q(w_2\otimes w_4)+P(w_2)t^2+P(w_3),\,Q(w_3\otimes w_4)+P(w_3)t,\, P(w_4)
\rangle.
\end{multline*}
Now we compute in the ideal $J:=\ker(\pi_1^*)$ with the multiplication property \eqref{eq : multiplication Q and t 2} in mind.
\begin{eqnarray*}
	  Q(w_0\otimes w_4)+Q(w_1\otimes w_3)+t^4+P(w_2)\in J &\Longrightarrow &   t^5+P(w_2)t \in J,\\
	  Q(w_2\otimes w_4)+P(w_2)t^2+P(w_3)\in J &\Longrightarrow &  Q(w_2\otimes w_4)+t^6+P(w_3)\in J,\\
	   Q(w_2\otimes w_4)+t^6+P(w_3)\in J &\Longrightarrow &   t^7+P(w_3)t\in J,\\
	  Q(w_3\otimes w_4)+P(w_3)t \in J\quad\text{and}\quad t^7+P(w_3)t \in J &\Longrightarrow & t^7+ Q(w_3\otimes w_4)\in J,\\
	 t^7+ Q(w_3\otimes w_4) \in J  &\Longrightarrow &  t^8\in J.
\end{eqnarray*}
Hence $t^7\notin J$ while $t^8\in J$.
Therefore
\[
\ind_{\Z/2}(G_{4}(\R^{8});\F_2)=\ker(\pi_1^*)\cap \im \B(p)^*=\langle t^8\rangle.
\]
This completes a proof of Theorem \ref{th : main - 02} in the case when $n=4$.
\end{example}

\section{Proofs of main theorems}
\label{sec : proofs }

In this section we prove both Theorem~\ref{th : main - 01} and Theorem~\ref{th : main - 02} along the same line of arguments.
Therefore we introduce a common notation.

Let $n\geq 2$ be an even integer, and let $a\geq 1$ and $b\geq 0$ be the unique integers such that $n=2^a(2b+1)$.
Recall that we saw $W= \SO (n)\wr \Z/2 $  as a subgroup of $\SO (2n)$ so that $ \widetilde{G}_n(\R^{2n})/(\Z/2) \cong \SO(2n)/W$.
Furthermore we denoted by $i\colon W \longrightarrow \SO(2n)$ the inclusion, and by $\B(i)\colon \B W \longrightarrow \B \SO (2n) $ the induced map.

For $n\geq 1$ an arbitrary positive integer, let $a\geq 0$ and $b\geq 0$ be integers such that $n=2^a(2b+1)$.
Recall that we interpreted the group $U= \OO (n)\wr \Z/2 $  as a subgroup of $\OO (2n)$ so that $G_n(\R^{2n})/(\Z/2) \cong \OO(2n)/U$.
As before, by abusing the notation, we denoted by $i\colon U \longrightarrow \OO(2n)$ the inclusion, and with $\B(i)\colon \B U \longrightarrow \B \OO (2n) $ the corresponding induced map.

Notice that we denote by $i$ two different inclusions.
This should not cause any confusion since these two maps play the same role in the corresponding proofs of Theorems \ref{th : main - 01} and \ref{th : main - 02} that proceed along the same lines.

\medskip
In order to ease the notation in the proofs of theorems we introduce a simplification of notation:
\begin{compactitem}
	\item  $\xi:=c^*\gamma^n(\R^{\infty})\text{  or  } \gamma^n(\R^{\infty})$,
	\item  $w_i := w_i(\xi )$,
	\item  $Q_{i,j}:=Q(w_i\otimes w_j):= w_i\otimes w_j + w_j\otimes w_i$,
	\item  $P_i:=P(w_i):= w_i\otimes w_i$
	\item  $X:= W \text {  or  } U \text{  respectively}$.
\end{compactitem}

\medskip
As we have seen in the Section \ref{sec : Cohomology of the group $W$}, for oriented Grassmann manifolds, and in the Section \ref{sec : Cohomology of the group U} for Grassmann manifolds, the $0$-column of the $E_2$-term of the Lyndon--Hochschild--Serre spectral sequence \eqref{eq : spectral sequence - 002} and \eqref{eq : spectral sequence - 02} is additively generated by appropriate
\[
\{ Q_{i,j}\text{ }|\text{ }0\leq i<j\leq n \}
\qquad\text{and}\qquad
\{P_i\text{ } |\text{ } 0\leq i\leq n\},
\]
where $E_2^{0,*} \cong H^*(\SO(n);\F_2)^{\Z/2}$ for the oriented Grassmann manifold, and $E_2^{0,*} \cong H^*(\OO(n);\F_2)^{\Z/2}$ in the case of the Grassmann manifold.
Furthermore, $w_0=1$ and consequently $P_0:=P(w_0)=1\otimes 1 = 1$.

\medskip
Recall that in \eqref{eq : kernel} and \eqref{eq : kernel 2} we have proved the following description of the kernel ideal
\[
	\ker(\pi_1^*)=\langle \B(i)^*(w_1),\ldots,\B(i)^*(w_{2n})\rangle =	\langle w_1(\xi\wr\Z/2),\ldots, w_{2n}(\xi\wr\Z/2)\rangle \subseteq H^*(\B X;\F_2) .	
\]
Notice when we work over the oriented Grassmann manifold $w_1=w_1(\xi)=0$ in $H^*(\SO(n);\F_2)$.
Therefore, $w_1(\xi\wr\Z/2)=\B(i)^*(w_1)=0$ and so $Q_{1,j}=Q_{j,1}=0$ and $P_1=0$.

\begin{proposition}
\label{prop : calculus}
Let $n\geq 2$ be an  even integer, and let $a\geq 1$ and $b\geq 0$ be the unique integers such that $n=2^a(2b+1)$.
For every $1\leq k \leq 2^{a+1}-1$ the relation
\[
w_k(\xi\wr\Z/2)=0
\]
in the quotient algebra $H^*(X; \F_2)/\langle w_1(\xi\wr\Z/2), w_2(\xi\wr\Z/2),\cdots, w_{k-1}(\xi\wr\Z/2) \rangle$ is
\begin{compactenum}[\rm (i)]

\item in the case when $k$ is odd and $1\leq k \leq 2^{a+1}-3$ equivalent to
\begin{equation}
\label{k odd}
\sum_{i=\max\{ 0,k-n\} }^{[\frac{k-1}{2}]} Q_{i,k-i} = 0,
\end{equation}

\item in the case when $k$ is even and $k\notin\{2^{a+1}-2^a,2^{a+1}-2^{a-1},\ldots , 2^{a+1}-2\}$ equivalent to
\begin{equation}
\label{k even1}
P_{\frac{k}{2}} = \sum_{i=\max\{ 0,k-n\}}^{[\frac{k-1}{2}]} Q_{i,k-i},
\end{equation}

\item in the case when $k=2^{a+1} - 2^{r+1}$, where $r\in\{0,1,\ldots, a-1\}$, equivalent to
\begin{equation}
\label{k even2}
P_{2^a-2^r} = t^{2^{a+1}-2^{r+1}}+ \sum_{i=\max\{0,k-n\}}^{2^a-2^r-1}Q_{i,k-i} ,
\end{equation}

\item in the case when $k=2^{a+1}-1$ equivalent to
\begin{equation}
\label{last}
t\cdot P_{2^a -1} = \sum_{i=\max\{ 0,k-n\}}^{2^a-1}Q_{i,k-i}.
\end{equation}
\end{compactenum}

\end{proposition}

\begin{proof}
For the proof we use equations \eqref{XXX} and \eqref{XXXY} which are equivalent to the following sequence of equalities
\begin{equation}
\label{sw class}
w_k(\xi\wr\Z/2) = \sum_{i=\max\{ 0,k-n\}}^{[\frac{k-1}{2}]}Q_{i,k-i} + \sum_{i=\max\{ 0,k-n\}}^{[\frac{k}{2}]}\binom{n-i}{k-2i}t^{k-2i}P_i.
\end{equation}
In computations of binomial coefficients we use the classical Lukas's theorem \cite{Lucas1978}:
Let $a=a_s2^s +a_{s-1}2^{s-1} + \cdots +a_0$ and $b= b_s2^s +b_{s-1}2^{s-1}+\cdots +b_0$ be $2$-adic expansions of the natural numbers $a$ and $b$, with $a_i,b_i \in \{0,1\}$, where $a_s = 1$, while $b_s$ might be equal to $0$. Then
\begin{equation}
\label{bin coef}
\binom{a}{b} \equiv 0\quad (\mathrm{mod}~2)\quad\text { if and only if }\quad a_i < b_i \text{ for some } 0\leq i\leq s.
\end{equation}
Now the proof of the proposition follows in several steps by considering several independent cases.

\smallskip
First we establish the claim in the case when $a=1$, that is for $n=2(2b+1)$.
In this case the equality \eqref{sw class}, for $1\leq k\leq 3=2^{a+1}-1$, implies  that
\[
Q_{0,1}=0, \qquad\qquad P_1=t^2 +Q_{0,2}, \qquad\qquad t\cdot P_1 = Q_{0,3}+Q_{1,2},
\]
as claimed by the proposition

\smallskip
Now let us assume that $n$ is fixed and $a\geq 2$, meaning $4\mid n$, and let $1\leq k\leq 2^{a+1}-1$.
We prove the proposition in this case using induction on $k$, where for the induction hypothesis we take that the claim is true for all values smaller than $k$.
We investigate several cases.

\begin{compactitem}
\item Let $k=1$.
      From \eqref{sw class} we get $Q_{0,1}+\binom{n}{1}t=0$.
      Since $n$ is even this simplifies to the equality $Q_{0,1}=0$, which is in compliance with~\eqref{k odd}.

\item Let $k=2$.
	  The equality~\eqref{sw class} now reads:
		\[
		w_2(\xi\wr\Z/2)= Q_{0,2} + \binom{n}{2}t^2 +\binom{n-1}{0}P_1=Q_{0,2}+P_1=0
		\]
		confirming ~\eqref{k even1}.

\item Let $k$ be odd, and let $k\leq 2^a-1$.
	  Assume that the claim holds for all values $1,2,\cdots , k-1$.	
	  Since $t\cdot Q_{i,j}=0$ for all values of $i$ and $j$, and, by inductive hypothesis, all $P_i$ are sum of $Q$'s, for $1\leq i\leq k-1$, we conclude that $t\cdot P_i=0$.	
	  Therefore, in the second sum of the formula \eqref{sw class} the only possibly non-zero summands are for $i=0$ or $k-2i=0$.
	  Because $k$ is odd the second possibility is not possible.
	  On the other hand the first summand is zero since for $i=0$ the binomial coefficient $\binom{n}{k}=0$.
	  Thus we obtained~\eqref{k odd}.
	
\item Let $k$ be even, and let $k\leq 2^a-2$.
	  With the same reasoning as in the case for $k$ odd, all summands in the second sum of the formula \eqref{sw class} are zero except when $i=0$ and $k-2i=0$.
	  For $i=0$ the coefficient is equal to $\binom{n}{k}$, but the last $a$ digits in dyadic expansion of $n$ are zeros, while this is not the case for $k$ since $k<2^a$.
	  Therefore, the coefficient $\binom{n}{k}$ vanishes.
	  For $k-2i=0$ the corresponding coefficient is $1$ and we have verified~\eqref{k even1}.	
	
\item Let $k=2^a$.
	  As before, all summands in the second sum of the formula \eqref{sw class} are zero except for $i=0$ and $k-2i=0$.
	  For $i=0$ the binomial coefficient is equal to $\binom{n}{k} = \binom{2^a(2b+1)}{2^a}=1$ with the corresponding summand equal to $t^{2^a}$.
	  For $k-2i=0$, that is $i=2^{a-1}$, the related summand is $P_{2^{a-1}}$.
	  Consequently, we get~\eqref{k even2} in the case $r=a-1$.
	
\item Let $k$ be odd, and let $2^a+1\leq k\leq 2^{a+1}-3$.
      Then there exists a unique integer $r\in \{ 1,2, \cdots , a-1\}$ such that $2^{a+1}-2^{r+1}+1\leq k\leq  2^{a+1} - 2^r-1$.
      Now in the second sum of the formula \eqref{sw class}, using inductive hypothesis and the fact that $t\cdot Q_{i,j}=0$, we conclude that possible non-zero summands are obtained only for $i\in \{ 0, 2^a-2^{a-1},\cdots , 2^a-2^{r+1}\}$.
      The binomial coefficients of those summands are equal to $\binom{n-2^a+2^j}{k-2^{a+1}+2^{j+1}}$, for some $j=r,r+1,\ldots ,a$.
      All of them vanish since $n-2^a+2^j$ is even and  $k-2^{a+1}+2^{j+1}$ is odd.
      So we get~\eqref{k odd}.

\item Let $k$ be even, $k\geq 2^a+2$ with $k\notin\{2^{a+1}-2,2^{a+1}-2^2,\ldots,2^{a+1}-2^{a-1}\}$.
	  Then there exists a unique integer $r\in \{ 0, 1, \cdots , a-1\}$ such that $2^{a+1}-2^{r+1}+2\leq k\leq 2^{a+1} - 2^r-2$.
	  As before, using inductive hypothesis, all summands from the second sum of the formula \eqref{sw class} which are possibly non-zero are obtained for $i\in \{ 0, \frac{k}{2}, 2^a-2^{a-1},\cdots , 2^a-2^r\}$.
	  We analyze three different cases.
	  	\begin{compactitem}
	  	\item For $i=0$ the binomial coefficient is $\binom{n}{k}$ and is equal to zero since $2^a \nmid k$.
	  	\item For $i=\frac{k}{2}$ the binomial coefficient is equal to $1$ while the corresponding summand is equal to $P_{\frac{k}{2}}$.
	   	\item For $i=2^a-2^j$, where $j\in\{r, r+1, \cdots ,a-1\}$, the binomial coefficient is equal to $\binom{n-2^a+2^j}{k-2^{a+1}+2^{j+1}}$.
	   		  In this case we have to examine dyadic expansion of the numbers involved more closely.
	   		  We get the following:
				\[
					n-2^a+2^j = n_1\ldots n_s0\ldots 010\ldots 0,
				\]
			  where $1$ is on the $(j+1)$-st place from the right.
			  It could happen that all $n_1, \ldots , n_s$ are equal to 0.
			  On the other hand since $2^{a+1}-2^{r+1}+2\leq k\leq 2^{a+1} - 2^r-2$ we have
			  \[
				k=1\ldots 10k_1\ldots k_r,
			  \]
			  with $a-r$ $1$'s and some $k_s=1$.
			  Consequently,
			  \[
				k-2^{a+1}+2^{j+1}= 1\ldots 10k_1\ldots k_r,
			  \]
			  with $j-r$ 1's and some $k_s=1$.
			  Therefore, some of the first $r$ coefficients from the right of the number $k-2^{a+1}+2^{j+1}$ is equal to $1$, while all of the first $r$ coefficients of the number $n-2^a+2^j$ are equal to $0$, and we may conclude that all binomial coefficients have to vasnih.
		\end{compactitem}
 	  Thus, we verified \eqref{k even1}.

\item 	Let $k=2^{a+1}-2^{r+1}$, and let $0\leq r\leq a-2$.
		As before, using inductive hypothesis, we conclude that all summands from the second sum of the formula \eqref{sw class} which have a chance to be  non-zero are obtained for $i\in \{ 0, \frac{k}{2}, 2^a-2^{a-1},\cdots , 2^a-2^r\}$.
		We discuss three separate cases.
		\begin{compactitem}
			\item	For $i=0$ the binomial coefficient is equal to $\binom{n}{k}$ and is equal to $0$, since dyadic expansion of $k$ is $k=1\ldots 10\ldots 0$, with $a-r$ ones and $r+1$ zeros.
			Therefore, $k$ has at least one non-zero entry in the first $a$ places from the right, while $n$ has only zeros.
			
			\item For $i=\frac{k}{2}=2^a-2^r$, the binomial coefficient is $1$ and the corresponding summand is  $P_{2^a-2^r}$.
			
			\item For $i=2^a-2^j$, where $j\in \{ r+1,\ldots ,a-1\}$, we have that the relevant  binomial coefficient is $\binom{n-2^a+2^j}{2^{a+1}-2^{r+1}-2^{a+1}+2^{j+1}}$.
			The dyadic expansion of the top entry in the binomial coefficient is of the form
\[
n-2^a+2^j = n_1\ldots n_s0\ldots 010\ldots 0,
\]
having $a-j$ and $j$ zeros, respectively.
On the other hand the lower entry has the following dyadic expansion
\[
2^{a+1}-2^{r+1}-2^{a+1}+2^{j+1}=2^{j+1}-2^{r+1}=1\ldots 10\ldots 0,
\]
with $j-r$ ones and $r+1$ zeros.
The only non-zero coefficient appears for $j=r+1$ with the summand $t^{2^{r+1}}\cdot P_{2^a-2^{r+1}}$.
Hence we have
\[
P_{2^a-2^r} = \sum_{i=\max\{ 0,k-n\}}^{2^a-2^r-1}Q_{i,k-i}+ t^{2^{r+1}}\cdot P_{2^a-2^{r+1}}
\]
Use inductive hypothesis for $P_{2^a-2^{r+1}}$ and get
\begin{multline*}
P_{2^a-2^r} =\sum_{i=\max\{ 0,k-n\}}^{2^a-2^r-1}Q_{i,k-i}+ t^{2^{r+1}}\left (t^{2^{a+1}-2^{r+2}}+ \sum_{i=\max\{ 0, k-n\}}^{2^a-2^{r+1}-1}Q_{i,2^{a+1}-2^{r+2}-i}\right )=\\
\sum_{i=\max\{ 0,k-n\}}^{2^a-2^r-1}Q_{i,2^{a+1}-2^{r+1}-i} + t^{2^{a+1}-2^{r+1}}.	
\end{multline*}

		\end{compactitem}
We completed the proof of \eqref{k even2}.

\item 	Let $k=2^{a+1}-1$.
		Since $k$ is odd in the second sum of the formula \eqref{sw class}, if $i$ is even, the binomial coefficient has to vanish since $n-i$ is even and $k-2i$ is odd.
		If $i$ is odd, all summands are zero unless $i=2^a-2^r$, for some $r\in \{ 0, 1, \ldots , a-1\}$, because they are of the form $t\cdot P_i$ and for such $i$, $P_i$ is a sum of $Q$'s.
		Now let $i=2^a-2^r$, for some $0\leq r \leq a-1$.
		In this case the binomial coefficient is equal to $\binom{n-2^a+2^r}{2^{a+1}-1-2^{a+1}+2^{r+1}}$ and dyadic expansions of the binomial coefficient are equal to:
\[
n-2^a+2^r = n_1\ldots n_s0\ldots 010\ldots 0,
\]
where there are $a-r$ and $r$ zeros, respectively; and
\[
2^{a+1}-1-2^{a+1}+2^{r+1}=2^{r+1}-1 = 1\ldots 1,
\]
with $r+1$ ones.
The only $i$ for which the coefficient is equal to $1$ is for $r=0$, that is $i=2^a-1$.
For such $i$, an appropriate summand is equal to $t\cdot P_{2^a-1}$, and we proved~\eqref{last}.
\end{compactitem}

Since we cover all possibilities we have completed the proof of the proposition.
\end{proof}

\begin{corollary}
\label{cor}
Let $n\geq 2$ be an even integer, and let $a\geq 1$ and $b\geq 0$ be the unique integers such that $n=2^a(2b+1)$.
In the quotient algebra $ H^*(X; \F_2)/\langle w_1(\xi\wr\Z/2), w_2(\xi\wr\Z/2),\ldots, w_{2^{a+1}-1}(\xi\wr\Z/2) \rangle $ holds
\[
t^{2^{a+1}} =0.
\]
\end{corollary}
\begin{proof}
We have
\[
t^{2^{a+1}}= t^2\cdot t^{2^{a+1}-2}= t^2\left ( t^{2^{a+1}-2} +\sum_{i=\max\{ 0,2^{a+1}-2-n\}}^{2^a-2}Q_{i,2^{a+1}-2-i}\right ).
\]
Now we use~\eqref{k even2}, for $r=0$, and~\eqref{last} and get
\[
t^{2^{a+1}} = t^2\cdot P_{2^a-1} = t \left ( \sum_{i=\max \{ 0,2^{a+1}-n-1\}}^{2^a-1}Q_{i,2^{a+1}-1-i}\right )=0.
\]
\end{proof}

\medskip
Finally we prove Theorem~\ref{th : main - 01} and Theorem~\ref{th : main - 02}.

\begin{proof}[Proof of Theorem~\ref{th : main - 01}]
We consider three different cases depending on value of integer $a\geq 0$.

\smallskip
Let $n\geq 1$ be an odd integer, that is $a=0$.
Since $n$ is a multiple of $1$ from Proposition \ref{prop : Z/2-map between Grassmannians} we get a $\Z/2$-map $\widetilde{g}\colon \widetilde{G}_{1}(\R^{2})\longrightarrow \widetilde{G}_{n}(\R^{2n})$.
Furthermore, applying Proposition \ref{prop : covering map a Z/2-map} we have an additional $\Z/2$-map $c\colon \widetilde{G}_{n}(\R^{2n})\longrightarrow G_{n}(\R^{2n})$.
The composition map
\[
\xymatrix{
\widetilde{G}_{1}(\R^{2})\ar[r]^-{\widetilde{g}} & \widetilde{G}_{n}(\R^{2n})\ar[r]^-{c} & G_{n}(\R^{2n}),
}
\]
and the Corollary \ref{cor : index for odds} yield the following relation between the indexes
\[
\ind_{\Z/2}(\widetilde{G}_{1}(\R^{2});\F_2) \supseteq \ind_{\Z/2}(\widetilde{G}_{n}(\R^{2n});\F_2) \supseteq \ind_{\Z/2}(G_{n}(\R^{2n});\F_2) = \langle t^2\rangle.
\]
Since $\widetilde{G}_{1}(\R^{2})\cong\SO(2)$ is a sphere $S^1$ equipped with a free $\Z/2$ index we have that $\ind_{\Z/2}(\widetilde{G}_{1}(\R^{2});\F_2)=\langle t^2\rangle$.
Thus, we concluded the proof of Theorem \ref{th : main - 01} for $n$ odd by showing that
\[
\ind_{\Z/2}(\widetilde{G}_{n}(\R^{2n});\F_2)=\langle t^2\rangle.
\]

\smallskip
Next, let $a=1$, that is $n=2(2b+1)$.
From Proposition \ref{prop : calculus} with $k=2$, using the fact that $w_1=0$ and consequently  $Q_{1,j}=0$, for all $j$, and $P_1=0$, we have that $t^2 = Q_{0,2}$.
Since there are no other relations for $Q_{0,2}$ we conclude that $t^2 \neq 0$ and that $t^3 = tQ_{0,2}=0$. Therefore,
\[
\ind_{\Z/2}(\widetilde{G}_{n}(\R^{2n});\F_2) = \langle t^3 \rangle.
\]

\smallskip
Now consider the case $a\geq 2$.

First note that from Proposition~\ref{prop : Z/2-map between Grassmannians} and using Corollary~\ref{cor} we get
\[
\ind_{\Z/2}(\widetilde{G}_{2^a}(\R^{2^{a+1}});\F_2) \supseteq \ind_{\Z/2}(\widetilde{G}_{n}(\R^{2n});\F_2)  \supseteq \langle t^{2^{a+1}}\rangle.
\]
Therefore, it is enough to prove the Theorem~\ref{th : main - 01} for $n=2^a$. 

Let $n=2^a$.
As we know the index is generated by the first power of $t$ which is equal to zero in the quotient
\[
H^*(W; \F_2)/\langle w_2(\xi\wr\Z/2),\ldots, w_{2^{a+1}}(\xi\wr\Z/2) \rangle.
\]
What is left to be proved is that
\[
t^{2^{a+1}-1} \notin \langle  w_2(\xi\wr\Z/2),\ldots, w_{2^{a+1}}(\xi\wr\Z/2) \rangle.
\]
In order to simplify calculation we will expand the ideal with more generators and show that  $t^{2^{a+1}-1}$ is not in the bigger ideal.
Let
\[
I:= \langle \{ w_2(\xi\wr\Z/2),\ldots, w_{2^{a+1}}(\xi\wr\Z/2)\}\cup \{ Q_{i,j}| 1 \leq i+j \leq 2^{a+1}-2 \}\rangle.
\]
In the quotient $H^*(W; \F_2)/I$ we have from Proposition \ref{prop : calculus} that relations from ideal $I$ are equivalent with the following
\begin{eqnarray*}
P_k=0, &\qquad & \text{ for $1\leq k \leq 2^a-1$ and } k\notin \{ 2^a-2^{a-1},2^a-2^{a-2},\ldots ,2^a-1\},\\
P_k=t^{2k}, &\qquad &\text{ for $1\leq k\leq 2^a-1$ and } k\in \{ 2^a-2^{a-1},2^a-2^{a-2},\ldots ,2^a-1\},
\end{eqnarray*}
and
\[
t\cdot P_{2^a -1} = Q_{2^a-1,2^a}.
\]
Now from~\eqref{XXXY} and $w_{2^{a+1}}=0$ we get
\[
P_{2^a}=0.
\]
It follows that all $t^k\neq 0$ for $1\leq k \leq 2^{a+1}-2$ including $t^{2^{a+1}-2} = P_{2^a-1}$ and that
\[
t^{2^{a+1}-1} = t\cdot P_{2^a -1} = Q_{2^a-1,2^a}.
\]
Since this is the only relation involving $Q_{2^a-1,2^a}$, we get that the right hand side is not equal to 0. So $t^{2^{a+1}-1} \notin I$ and it does not belong in the smaller ideal $\langle w_2(\xi\wr\Z/2),\ldots, w_{2^{a+1}}(\xi\wr\Z/2) \rangle$ as well.

Therefore, using Corollary~\ref{cor} we get
\[
\ind_{\Z/2}(\widetilde{G}_{2^a}(\R^{2^{a+1}});\F_2) = \langle t^{2^{a+1}} \rangle.
\]
This concludes the proof of Theorem~\ref{th : main - 01}.
\end{proof}

\begin{proof}[Proof of Theorem~\ref{th : main - 02}]
The proof is given in three steps depending on the value of integer $a$.
For $a=0$, that is $n$ odd, the statement of the theorem is established in Corollary \ref{cor : index for odds}.

\smallskip
Let $a=1$, that is $n=2(2b+1)$.
First note that from Corollary~\ref{cor} we have that $t^4=0$ that is
\[
\ind_{\Z/2}(G_{n}(\R^{2n});\F_2)  \supseteq \langle t^{4}\rangle.
\]
On the other hand, from Proposition~\ref{prop : Z/2-map between Grassmannians} and Proposition~\ref{prop : index for n=1 and n=2}, or Example \ref{ex : n=2}, we get
\[
\ind_{\Z/2}(G_{n}(\R^{2n});\F_2)  \subseteq  \ind_{\Z/2}(G_{2}(\R^{4});\F_2)=\langle t^{4}\rangle.
\]
Therefore
\[
\ind_{\Z/2}(G_{n}(\R^{2n});\F_2)=\langle t^{4}\rangle.
\]

\smallskip
For the case $a\geq 2$ the proof is exactly the same as in the proof of the Theorem~\ref{th : main - 01} but we present it nevertheless.

First note that from Proposition~\ref{prop : Z/2-map between Grassmannians} and using Corollary~\ref{cor} we get
\[
\ind_{\Z/2}(G_{2^a}(\R^{2^{a+1}});\F_2) \supseteq \ind_{\Z/2}(G_{n}(\R^{2n});\F_2)  \supseteq \langle t^{2^{a+1}}\rangle.
\]
Therefore, it is enough to prove the Theorem~\ref{th : main - 01} for $n=2^a$. 

Let $n=2^a$.
The index is generated by the first power of $t$ which is equal to zero in the quotient
\[
H^*(U; \F_2)/\langle w_1(\xi\wr\Z/2),\ldots, w_{2^{a+1}}(\xi\wr\Z/2) \rangle.
\]
Thus, what is left to be proved is that
\[
t^{2^{a+1}-1} \notin \langle  w_1(\xi\wr\Z/2),\ldots, w_{2^{a+1}}(\xi\wr\Z/2) \rangle.
\]
In order to simplify calculation we will expand the ideal with more generators and show that  $t^{2^{a+1}-1}$ is not in the bigger ideal.
Let
\[
I:= \langle \{ w_1(\xi\wr\Z/2),\ldots, w_{2^{a+1}}(\xi\wr\Z/2)\}\cup \{ Q_{i,j}| 1 \leq i+j \leq 2^{a+1}-2 \}\rangle.
\]
In the quotient $H^*(U; \F_2)/I$ we have from Proposition \ref{prop : calculus} that relations from ideal $I$ are equivalent with the following
\begin{eqnarray*}
P_k=0, &\qquad & \text{ for $1\leq k \leq 2^a-1$ and } k\notin \{ 2^a-2^{a-1},2^a-2^{a-2},\ldots ,2^a-1\},\\
P_k=t^{2k},  &\qquad & \text{ for $1\leq k\leq 2^a-1$ and } k\in \{ 2^a-2^{a-1},2^a-2^{a-2},\ldots ,2^a-1\},	
\end{eqnarray*}
and
\[
t\cdot P_{2^a -1} = Q_{2^a-1,2^a}.
\]
Next from~\eqref{XXXY} and $w_{2^{a+1}}=0$ we get
\[
P_{2^a}=0.
\]
It follows that all $t^k\neq 0$ for $1\leq k \leq 2^{a+1}-2$ including $t^{2^{a+1}-2} = P_{2^a-1}$ and that
\[
t^{2^{a+1}-1} = t\cdot P_{2^a -1} = Q_{2^a-1,2^a}.
\]
Since this is the only relation involving $Q_{2^a-1,2^a}$, we get that the right hand side is not equal to 0. So $t^{2^{a+1}-1} \notin I$ and it does not belong in the smaller ideal $\langle w_1(\xi\wr\Z/2),\ldots, w_{2^{a+1}}(\xi\wr\Z/2) \rangle$ as well.
Therefore, using Corollary~\ref{cor} we get
\[
\ind_{\Z/2}(G_{2^a}(\R^{2^{a+1}});\F_2) = \langle t^{2^{a+1}} \rangle.
\]
This concludes the proof of Theorem~\ref{th : main - 02}.
\end{proof}

\begin{small}
\providecommand{\noopsort}[1]{}
\providecommand{\bysame}{\leavevmode\hbox to3em{\hrulefill}\thinspace}
\providecommand{\MR}{\relax\ifhmode\unskip\space\fi MR }
\providecommand{\MRhref}[2]{%
  \href{http://www.ams.org/mathscinet-getitem?mr=#1}{#2}
}
\providecommand{\href}[2]{#2}

\end{small}
\end{document}